\documentclass[a4paper,11pt]{article}

\usepackage{enumerate}
\usepackage{caption}
\usepackage{amsmath}
\usepackage{amsthm}
\usepackage{mathrsfs}
\usepackage{epsfig}
\usepackage{verbatim}
\usepackage[T1]{fontenc}
\usepackage{amsfonts}
\usepackage{float}
\usepackage{caption}
\usepackage{graphicx}
\usepackage{color}
\usepackage{fullpage}

\usepackage[latin2]{inputenc}
\usepackage{url}

\newtheorem{theorem}{Theorem}[section]

\newtheorem{proposition}[theorem]{Proposition}
\newtheorem{lemma}[theorem]{Lemma}
\newtheorem{corollary}[theorem]{Corollary}
\newtheorem{definition}[theorem]{Definition}

\newtheorem{remark}[theorem]{Remark}

\makeindex

\begin{document}
\title{On uniqueness of dissipative solutions of the Camassa-Holm equation\\
\small{}}
\author{{Grzegorz Jamr\'oz} \\
 {\it \small  University of Warwick, Mathematics Institute,}\\
{\it \small Zeeman Building, Coventry CV4 7AL, United Kingdom}\\
{\it \small e-mail: G.Jamroz@warwick.ac.uk}
}

\maketitle

\abstract{
We prove that dissipative weak solutions of the Camassa-Holm equation are unique. Thus we complete the global well-posedness theory of this celebrated model of shallow water, initiated by a general proof of existence in [Z. Xin, P. Zhang \emph{Comm. Pure Appl. Math.} {\bf 53} (2000)]. As the dissipative weak solutions, being viscosity solutions, seem to constitute a physically relevant class of solutions  in the wave-breaking regime, the result provides mathematical rationale for the feasibility of the Camassa-Holm equation as a model of water waves encompassing both soliton interactions and wave-breaking. 
}
{\bf MSC Classification 2010:} 35L65, 37K10

\section{Introduction}
Mathematical theory of shallow water has  been of constant interest to the fluid mechanics community as it can serve to model various important phenomena occuring in rivers, ponds, continental shelfs etc. The general Euler equations of fluid dynamics being rather intractable, the quest has been to find simplified equations which would capture the most important observable phenomena.   
One of the landmarks herein constitute the experiments of John Scott Russell in 1834 on shallow canals, which, somewhat accidentally, brought him to the discovery of solitons -- long waves traveling over large distances virtually without changing their shape and whose shape is preserved after interaction with another wave of the same kind. This led, after many years of theoretical work by, among others, Lord Rayleigh, Joseph Boussinesq, Diederik Korteweg and Gustav de Vries, to the discovery of the Korteweg de Vries (KdV) equation, 
\begin{equation*}
u_t + 6uu_x + u_{xxx}=0,
\end{equation*}
which became the most important model of solitons and has found applications not only in the shallow water theory but also in other areas of physics such as physics of plasma or crystallography.

It soon became clear, however, that in the realm of shallow waters there is another observable phenomenon, which has to be accounted for by a feasible mathematical model -- the wave-breaking, which consists in the wave profile developing a vertical slope in finite time while preserving bounded amplitude. As the KdV equation exhibits global smooth solutions  for smooth intial data (see \cite{KPV}), new equations modeling the wave-breaking (yet lacking the soliton solutions), such as the Whitham equation, (see \cite{Whitham, CE3}) were proposed.
This was, however, far from satisfactory and finally in 1993, after a long search, a new model encompassing both phenomena -- solitons and wave-breaking, was derived in \cite{CH} by using an asymptotic expansion in the hamiltonian for the vertically averaged incompressible Euler equations (see discussion in \cite{CHH}). The model, nowadays known as the Camassa-Holm equation, reads
\begin{equation}
\label{Eq_CHoriginall}
u_t - u_{xxt} + 3uu_x = 2u_x u_{xx} + uu_{xxx},
\end{equation}
where $t$ denotes time, $x$ is a one-dimensional space variable and $u(t,x)$ corresponds to the horizontal velocity of water surface (see \cite{CL}). The theory of smooth solutions for \eqref{Eq_CHoriginall} is due to A. Constantin and J. Escher \cite{CE1,CE2}, who also formulated fairly general conditions precluding wave-breaking (\cite{CE3,CE4}), see also \cite{Brandolese} for recent developments. In general, however, wave-breaking, understood as blow-up of the supremum norm of the derivative $u_x$ does occur, as predicted by Camassa and Holm in \cite{CH}, and there is no way to prevent it. This means that the global well-posedness theory has to be sought for in the more general space of weak (i.e. non-smooth) solutions, which are no longer three times differentiable. A convenient framework for this purpose is provided by the system
\begin{eqnarray}
\partial_t u + \partial_x (u^2 / 2) + P_x &=& 0, \label{eq_WeakCH1}\\
P(t,x) &=& \frac 1 2 e^{-|x|}* \left(u^2(t,\cdot) + \frac {u_x^2(t,\cdot)}{2}\right),\\
u(t=0, \cdot) &=& u_0, \label{eq_WeakCH3}
\end{eqnarray}
with $*$ denoting the convolution, which for smooth solutions is equivalent to \eqref{Eq_CHoriginall}, with equivalence  provided through the nonlocal operator $(I-\partial_{xx})^{-1}$, which satisfies $(I-\partial_{xx})(\frac 1 2 e^{-|x|}) = \delta(x)$. Note that for \eqref{eq_WeakCH1}-\eqref{eq_WeakCH3} to make sense it suffices to assume $u \in L^{\infty}([0,\infty),H^1(\mathbb{R}))$, which corresponds to the maximal physically relevant class of solutions with bounded total energy, which is given by
$$E(t) = \frac 1 2 \int_{\mathbb{R}} (u^2(t,x)+u_x^2(t,x)) dx.$$
Importantly, the total energy is conserved for smooth solutions (see e.g.  Section 2 of \cite{BC2}). The conservation of energy, however, does not hold in general for weak solutions, which allows one to encompass, for instance, the phenomena related to wave breaking, in which the energy is dissipated.

\begin{definition}[Weak solutions]
\label{Def_Weak}
Let $u_0 \in H^1(\mathbb{R})$. We say that a function $u: [0,\infty) \times \mathbb{R} \to \mathbb{R}$ is a \emph{weak solution} of \eqref{eq_WeakCH1}-\eqref{eq_WeakCH3} if
\begin{itemize}
\item $u(t,x) \in C([0,\infty) \times \mathbb{R}) \cap L^{\infty}([0,\infty), H^1(\mathbb{R})),$
\item $u(t=0,x) = u_0(x)$ for $x \in \mathbb{R}$,
\item $u(t,x)$ satisfies \eqref{eq_WeakCH1} in the sense of distributions,
\end{itemize}
where $H^1(\mathbb{R})$ is the Sobolev space with the norm $\|f\|_{H^1(\mathbb{R})}^2 = \int_{\mathbb{R}} (f^2(x) + f_x^2(x))dx$, $C([0,\infty) \times \mathbb{R})$ denotes the space of continuous functions on $[0,\infty) \times \mathbb{R}$ and $L^{\infty}([0,\infty), H^1(\mathbb{R}))$ stands for essentially bounded measurable functions from the time interval $[0,\infty)$ into $H^1(\mathbb{R})$.
\end{definition}
Existence of weak solutions of the Camassa-Holm equation was proven by Xin and Zhang in \cite{XZ} (see also \cite{CE2} for a previous conditional result) by the method of vanishing viscosity, which resulted in the so-called \emph{dissipative weak solutions}, which are a counterpart of the entropy solutions known from the theory of hyperbolic conservation laws.

\begin{definition}[Dissipative weak solutions]
\label{Def_dissipative}
A weak solution of \eqref{eq_WeakCH1}-\eqref{eq_WeakCH3} is called \emph{dissipative} if
\begin{itemize}
\item $\partial_x u(t,x) \le const \left(1+ \frac 1 t\right)$  (Oleinik-type condition)
\item $\|u(t,\cdot)\|_{H^1(\mathbb{R})} \le  \|u(0,\cdot)\|_{H^1(\mathbb{R})}$ for every $t > 0$ (weak energy condition).
\end{itemize}
\end{definition}

Weak solutions of \eqref{eq_WeakCH1}-\eqref{eq_WeakCH3} are by no means unique.  On top of dissipative weak solutions, there exist, for instance, classes of conservative  \cite{BC} solutions, which conserve locally the energy, and intermediate \cite{GHR2} solutions, which interpolate between the conservative and dissipative ones. The conservative solutions were shown to be unique (in a somewhat restricted space in comparison to Definition \ref{Def_Weak}, however) by Bressan, Chen and Zhang \cite{BCZ} and in  \cite{BF,GHR, GHR1, HR2} distance functionals yielding conditional uniqueness results were constructed. The uniqueness of dissipative solutions, on the other hand, has been very elusive and, although anticipated (see e.g. the comment after inequality (1.2) in \cite{BC2} or Section 4 in \cite{BressanB}), never proved. It has thus remained, despite over 15 years of interest of leading experts in the field, an outstading open problem.
Apart from the conditional weak-strong uniqueness result in \cite{XZ2}, the only available results in this direction have been constructions of specific global semigroups of dissipative solutions (\cite{BC2,HR}). As remarked, however, by Bressan and Constantin in \cite{BC2} such a construction does not exclude the possibility that other constructive procedures yield distinct dissipative solutions and so the issue of uniqueness is rather delicate.

In this paper we provide an affirmative answer to the question of uniqueness.
\begin{theorem}
\label{th_DU}
Let $u_0 \in H^1(\mathbb{R})$. There exists a unique dissipative weak solution of \eqref{eq_WeakCH1}-\eqref{eq_WeakCH3} with initial condition $u_0$.
\end{theorem}
Thus we demonstrate the second major ingredient of the well-posedness of the Camassa-Holm equation in the sense of Hadamard (see Table \ref{Tab1}), the first being existence of solutions proven by Xin and Zhang. The third ingredient -- continuity of solutions with respect to perturbation of initial data -- follows immediately (provided Theorem \ref{th_DU} is satisfied) from the theory of Bressan-Constantin \cite[Theorem 8.1]{BC2}. More precisely, \cite[Theorem 8.1]{BC2} asserts that if $\|u_0^n - u_0\|_{H^1(\mathbb{R})} \rightarrow 0$ then the corresponding solutions $u_n(t,x)$ converge to $u(t,x)$ uniformly for $t,x$ in bounded sets. 
\begin{table}[h!]
\begin{tabular}{|l|l|}
  \hline 
  Existence & Xin-Zhang \cite{XZ}\\
  \hline
  Uniqueness  & {\bf Theorem \ref{th_DU}} \\
  \hline
  Continuity with respect to initial data  & Bressan-Constantin \cite{BC2} provided uniqueness holds
  \\
  \hline
\end{tabular} 
\caption{Well-posedness of dissipative weak solutions of the Camassa-Holm equation. The existence was proven in 2000 by Xin and Zhang. A conditional result on continuity with respect to initial data (under assumption that Theorem \ref{th_DU} holds) was proven in 2007 by Bressan and Constantin. Finally, Theorem \ref{th_DU}, using the theory of Bressan-Constantin from \cite{BC2} resolves the problem of uniqueness.}
\label{Tab1}
\end{table}
We conclude that the Camassa-Holm equation is well-posed in the sense of Hadamard, which confirms feasibility of the equation even after wave breaking. Let us here emphasize that due to the fact that dissipative solutions are obtained by the vanishing viscosity method and thus correspond to entropy solutions, they presumably constitute the most physically relevant class of solutions.
Let us also mention that uniqueness of dissipative weak solutions implies that \emph{every} such solution can be obtained by the construction of Bressan-Constantin in \cite{BC2} and, in particular, every dissipative solution has better regularity properties assumed in \cite[Definition 2.1]{BC2}:
\begin{itemize}
\item H\"older continuity,
\item Lipschitz continuity of $t \mapsto u(t,\cdot)$ from $[0,\infty)$ into $L^2(\mathbb{R})$,
\item satisfaction of the equality $\frac {d}{dt} u = -uu_x - P_x$ for a.e. $t \in [0,\infty)$ where the equality is understood as identity of functions in $L^2(\mathbb{R})$.

\end{itemize}

To prove Theorem \ref{th_DU} we show that dissipative weak solutions can be completely described in terms of a representation calculable along characteristics. It turns out that this representation, after suitable modification on a negligible set,  satisfies the system of ordinary differential equations used by Bressan and Constantin in \cite{BC2} to construct a global semigroup of dissipative solutions. As the solutions of that system are unique (see \cite{BC2}), so are dissipative weak solutions.

The paper builds upon the research  in \cite{CHGJ}, where a framework for studying weak solutions of the Camassa-Holm and related models by means of nonunique characteristics was proposed. The setup of the framework follows the ideas from \cite{TCGJ} where, together with T. Cie\'slak, we developed a similar framework in the context of the related Hunter-Saxton equation. This led us to a positive verification of the hypothesis due to Zhang and Zheng \cite{ZhangZheng} that maximal dissipation criterion selects the unique dissipative solutions out of the collection of weak solutions of the Hunter-Saxton equation. 
The paper \cite{TCGJ}, in turn, was inspired by the ideas of C. Dafermos dating back to the 1970s (\cite{DafGC}), where the possibility of studying solutions of hyperbolic conservation laws by means of nonunique characteristics was demonstrated. These ideas have recently been revisited and applied by Dafermos in the context of Hunter-Saxton equation \cite{DafHS} and also by Bressan, Chen and Zhang  in \cite{BCZ} to study uniqueness of conservative solutions of the Camassa-Holm equation. 

To conclude, let us give a few more detailed remarks regarding the proof of Theorem \ref{th_DU}, which, in particular, highlight some of the new challenges in comparison to related previous results.

\begin{enumerate}[i)]
\item
The strategy of proof of Theorem \ref{th_DU} resembles in principle the strategy developed by Dafermos in \cite{DafHS} for the related, yet considerably simpler, Hunter-Saxton equation and is to some extent similar to the strategy used in \cite{BCZ} in the case of conservative solutions of the  Camassa-Holm equation. The main new challenge  consists in showing that almost every $(t,x) \in [0,\infty) \times \mathbb{R}$ lies on a 'good' characteristic curve, along which certain ordinary differential equation for $u_x$ is satisfied. In \cite{DafHS} this property was demonstrated a posteriori, after first obtaining an explicit representation formula for solutions, which is not possible for the Camassa-Holm equation. In \cite{BCZ}, on the other hand, this property followed from the properties of the well-thought-out change of variables $(t,x) \to (t,\beta)$, which transformed the solution $u(t,x)$ into a Lipschitz continuous function $u(t,\beta)$, and allowed the authors to use  directly the framework from \cite{DafHS}. 
In the dissipative case such a change of variables being unavailable, we stick to the original variables $(t,x)$ and resolve this issue (see Theorem \ref{Th_STT} in Section \ref{Sec_crucial}) by considering for every timepoint $T>0$ the backward solution and characteristics emanating from time $T$ backwards in time. As the backward solution is no longer dissipative, the general theory of weak solutions of the Camassa-Holm equation developed in \cite{CHGJ} seems to be necessary. Finally, let us emphasize that the approach with consideration of backward characteristics seems more general than previous methods and gives better insights into the structure of solutions, especially in applications involving general weak (not dissipative) solutions.

\item For the Hunter-Saxton equation the wave-breaking time along every given characteristic is determined solely by the value $u_0'$ at the starting point of this characterstic. For the nonlocal Camassa-Holm equation, on the other hand, the breaking time depends on the whole evolution in time and there is no explicit representation formula for dissipative solutions. This means in particular that only estimates of breaking time are available, which leads to further complications such as e.g. the availability of change of variables formulas in the Stjeltjes integral only on certain 'good' sets, see Proposition \ref{Prop_ChangeOfV}. In the case of conservative solutions these complications were resolved by the change of variables in \cite{BCZ}, such a change of variables being however, as mentioned before, unavailable in the case of dissipative solutions.  
\item Some of the results used in the present paper (see Section \ref{Sec_ConvLemmas}) were originally inspired and derived as simple consequences of the more general theory developed in \cite{CHGJ2} (which, in turn, relies on the second part of \cite{CHGJ}). Nevertheless, to make the proof of uniqueness as simple and self-contained as possible, we decided to present alternative proofs in special cases tailored specifically to the needs of the present paper. Thus we circumvent the necessity of resorting to the more complicated theory from \cite{CHGJ2}.
\end{enumerate}

\noindent The structure of the paper is the following. In the next section we review the framework of nonunique characteristics in the context of the Camassa-Holm equation introduced in the first part of \cite{CHGJ}. Section \ref{Sec_ConvLemmas} is devoted to proving some useful convergence results for dissipative weak solutions whereas in Section \ref{Sec_backward} we discuss the notion of backward solution. In Section \ref{Sec_crucial} we prove the key technical result on representation of dissipative solutions by means of characteristics. Finally, in Section \ref{Sec_MainProof} we present the construction of a semigroup of dissipative solutions from \cite{BC2} and prove Theorem \ref{th_DU}, showing that every weak dissipative solution coincides with the solution constructed by Bressan and Constantin.

{\bf Acknowledgements.} 
This research was supported by the EU Grant no. 706283 VORTSHEET.
The author would like to express his gratitude to Tomasz Cie\'slak, from the Institute of Mathematics, Polish Academy of Sciences in Warsaw, for helpful discussions regarding the paper. 

\section{Characteristics and Camassa-Holm}
\label{Sec_Mresults}
In this section we introduce the notation and present some of the key results of the framework developed in \cite{CHGJ}. 
The first proposition (Proposition 2.4 from \cite{CHGJ}), which is a straightforward consequence of \cite[Lemma 3.1]{DafHS} provides the existence of (nonunique) characteristics.  
\begin{proposition}
\label{Prop_Daf}
Let $u$ be a weak solution of \eqref{eq_WeakCH1}-\eqref{eq_WeakCH3}. Then for every $\zeta \in \mathbb{R}$ and every initial timepoint $t_0 \ge 0$ there exists a (nonunique) characteristic of $u$ emanating from $\zeta$, i.e. a function $\zeta : [t_0,\infty) \to \mathbb{R}$, which satisfies:
\begin{itemize}
\item $\zeta(t_0)=\zeta$,
\item $\frac d {dt} \zeta(t) = u(t,\zeta(t)),$
\item $\frac d {dt} u(t,\zeta(t)) = -P_x(t,\zeta(t)).$
\end{itemize}
\end{proposition}
\begin{remark}
Note that $P_x = -\frac 1 2 sgn(x)e^{-|x|} * (u^2 + \frac 1 2 u_x^2)$ and hence, by Young's inequality and boundednes of energy of $u$ we have $\|P_x\|_{L^\infty([0,\infty) \times \mathbb{R})}<\infty$. Consequently, $t \mapsto u(t,\zeta(t))$ is a Lipschitz continuous function and $t \mapsto \zeta(t)$ is continuously differentiable with Lipschitz continuous derivative.
\end{remark}

In the following, for a given weak solution $u$ of \eqref{eq_WeakCH1}-\eqref{eq_WeakCH3} by $\zeta(\cdot)$ we will denote any characteristic satisfying the conditions from Proposition \ref{Prop_Daf}.
 Similarly, $\eta(\cdot)$ will be a characteristic satisfying $\eta(t_0) = \eta$, $\frac d {dt}\eta(t) = u(t,\eta(t))$, $\frac d {dt} u(t,\eta(t)) = -P_x(t,\eta(t))$. If not specified otherwise, $t_0$ will be considered to be equal $0$.  
 
 A characteristic $\zeta(\cdot)$ is, by definition, \emph{unique forwards} if, for given $u$, $t_0$, $\zeta$ there exists exactly one function $\zeta(\cdot)$ satisfying the properties from Proposition \ref{Prop_Daf}. Note that existence of at least one such function is guaranteed by Proposition \ref{Prop_Daf}. We will say that a characteristic $\zeta(\cdot)$ is \emph{unique forwards} on $[t_0,T]$ if all characteristics $\zeta(\cdot)$ of $u$ satisfying $\zeta(t_0)=\zeta$ coincide, when restricted to time interval $[t_0,T]$. Clearly, uniqueness forwards on $[t_0,T]$ does not imply uniqueness forwards on $[t_0,T_2]$ for $T_2>T$.  A characteristic $\zeta(\cdot)$ is \emph{unique backwards} on $[t_0,T]$ if for every $\eta \in \mathbb{R}$ and any characteristic $\eta(\cdot)$ of $u$ the equality $\eta(T) = \zeta(T)$ implies $\eta(s) = \zeta(s)$ for every $s \in [t_0,T]$. 
 
Next, we define the 'difference quotient of $u$ along characteristics' by $$\omega(s):=\frac{u(s,\eta(s))-u(s,\zeta(s))}{\eta(s)-\zeta(s)},$$ where $\zeta(\cdot), \eta(\cdot)$ are any characteristics as discussed above with $t_0$ either given explicitly or equal $0$ otherwise.
Usually, $\zeta(\cdot)$ will be unique forwards, however $\eta(\cdot)$ not necessarily. In such a case, the computations involving $\omega$ will hold for every fixed $\eta(\cdot)$. Let us also remark that boundedness of $\|u(t,\cdot)\|_{H^1(\mathbb{R})}$ for every $t$ implies that $u(t,\cdot)$ is absolutely continuous for every $t \ge 0$. In particular, $u(t,x_2) - u(t,x_1) = \int_{x_1}^{x_2} u_x(t,x)dx$ for every $x_1,x_2 \in \mathbb{R}$ and every $t \in [0,\infty)$ and so 
$\omega(s) = \frac{u(s,\eta(s))-u(s,\zeta(s))}{\eta(s)-\zeta(s)} = \frac{\int_{\zeta(s)}^{\eta(s)}u_x(s,y)dy}{\eta(s)-\zeta(s)}.$ This implies that if $\zeta(s)$ is a Lebesgue point of $u_x(s,\cdot)$ then $\omega(s)$ converges to $u_x(s,\zeta(s))$ as $\eta(s) \to \zeta(s)$.

Finally, $\mathcal{L}^1$ stands for the one-dimensional Lebesgue measure.

Let us now introduce one of the key technical definitions of the framework  (\cite[Definition 5.8]{CHGJ}) 
\begin{definition}
\label{Def_Lt}
\begin{eqnarray*}
L_{T}^{unique,N}&:=& \{\zeta \in \mathbb{R}: \zeta(\cdot) \mbox{ is unique forwards on $[0,T]$, } \\ && \zeta(s) \mbox{ is a Lebesgue point of } u_x(s,\cdot) \mbox{ for almost every } s \in [0,T] \mbox{ and }\\ &&\forall_{\eta \in (\zeta- \frac 1 N, \zeta) \cup (\zeta,\zeta+ \frac 1 N),  s \in [0,T]} -N \le \omega(s) \le N \} ,\\
L_{T}^{unique} &:=& \bigcup_{N=1}^{\infty}L_{T}^{unique,N}.
\end{eqnarray*}
\end{definition}

The rationale behind this definition is contained in the following result, which states that along characteristics emanating from $L_{T}^{unique}$ not only $u$, as stated in Proposition \ref{Prop_Daf}, but also $u_x$ satisfies a certain ordinary differential equation.  
The formulation below is a combination of Proposition 2.5, Corollary 6.1 and Lemmas 4.1 and 4.2 from \cite{CHGJ}.

\begin{proposition}[Proposition 2.5, Corollary 6.1 from \cite{CHGJ}]
\label{Prop_MainCH}
Let $u$ be a weak solution of \eqref{eq_WeakCH1}-\eqref{eq_WeakCH3}. 
There exists a family of sets $\{S_T\}_{T>0}$, $S_T \subset \mathbb{R}$, such that for every $\zeta \in S_T$ the characteristic $\zeta(\cdot)$ is unique forwards and backwards on $[0,T]$ and for $0 \le t < T$
\begin{equation}
\label{Eq_PropL}
\dot{v}(t) = u^2(t) - \frac 1 2 v^2(t) - P(t) ,
\end{equation}
where $v(t) := u_x(t,\zeta(t))$, $u(t)=u(t,\zeta(t))$ and $P(t) = P(t,\zeta(t))$.
Moreover, $S_{T_1} \subset S_{T_2}$ for $T_1 > T_2$ and $\mathcal{L}^1(\mathbb{R} \backslash \bigcup_{T>0} S_T) = 0$. Finally, for every $\zeta \in S_T$ there exists $N>0$ such that $|v(\cdot)| \le N$ on $[0,T]$.
The sets $S_T$ can be chosen as
\begin{equation*}
S_T = L_T^{unique} \backslash Z_T,
\end{equation*}
where $Z_T \subset \mathbb{R}$ are sets of measure $0$. 
\end{proposition}

\begin{remark}
\label{Rem_Modification}
Equation \eqref{Eq_PropL} is satisfied in the sense:
\begin{equation}
\label{eq_6}
v(t)-v(0) = \int_0^t \left(u^2(s,\zeta(s)) - \frac 1 2 u_x^2(s,\zeta(s)) - P(s,\zeta(s))\right) ds
\end{equation}
for almost every $t \in (0,T]$, such that $\zeta(t)$ is a Lebesgue point of $u_x(t,\cdot)$. Hence, for fixed $t \in (0,T]$ \eqref{eq_6} is satisfied for almost every $\zeta(t)$, $\zeta \in S_T$. More precisely, this means that for fixed $t \in (0,T]$ the set $\{y \in \mathbb{R}: y = \zeta(t) \mbox{ for some } \zeta \in S_T \mbox{ and } \eqref{eq_6} \mbox { holds}\}$ is a full measure subset of $S_T(t) := \{\zeta(t): \zeta \in S_T\}$. Finally, once the sets $S_T$ are fixed, one can modify $u_x(t,x)$, for every $t > 0$, on a set of $x$ of Lebesgue measure $0$ in such a way that if $\zeta \in S_T$ then  \eqref{eq_6} holds for \emph{every} $t \in [0,T]$. 
Indeed, even though there is a continuum of parameters $T>0$, this modification involves only such $x$, which fail to be Lebesgue points of $u_x(t,\cdot)$, which is a set of measure $0$ for every $t>0$. Moreover, if $\zeta \in S_{T_1} \cap S_{T_2}$, $T_1<T_2$, then the modifications stemming from $T_1$ and $T_2$ are compatible, since they are defined so that \eqref{eq_6} holds, i.e.  $$u^{modified}_x(t,\zeta(t)) :=  u_x(0,\zeta) + \int_0^t \left(u^2(s,\zeta(s)) - \frac 1 2 u_x^2(s,\zeta(s)) - P(s,\zeta(s))\right) ds.$$  
\end{remark}

Next, let us recall the notion of the rightmost and leftmost characteristics (also called minimal and maximal characteristics, see e.g. \cite{DafGC}). Namely, given a weak solution $u$ of \eqref{eq_WeakCH1}-\eqref{eq_WeakCH3}, for every $t_0 \ge 0$ and $\zeta \in \mathbb{R}$ there exist (see \cite[Corollary 5.2]{CHGJ}) the \emph{rightmost} characteristic $\zeta^r(\cdot)$ and the \emph{leftmost} characteristic $\zeta^l(\cdot)$, which are the unique characteristics of $u$ defined on $[t_0,\infty)$ satisfying 
\begin{itemize}
\item $\zeta^r(t_0) = \zeta = \zeta^l(t_0),$
\item $\frac d {dt} \zeta^r(t) = u(t,\zeta^r(t)),$
\item $\frac d {dt} \zeta^l(t) = u(t,\zeta^l(t)),$
\item $\frac d {dt} u(t,\zeta^r(t)) = -P_x(t,\zeta^r(t)),$
\item $\frac d {dt} u(t,\zeta^l(t)) = -P_x(t,\zeta^l(t)),$
\item $\zeta^l(t) \le \zeta(t) \le \zeta^r(t)$ for $t \in [t_0,\infty)$ and every characteristic $\zeta(\cdot)$ of $u$ with $\zeta(t_0)=\zeta$.
\end{itemize}
The following important property is based on \cite[Section 7]{CHGJ}.
\begin{proposition}[Change of variables formula]
\label{Prop_ChangeOfV}
Let $g$ be a bounded nonnegative Borel measurable function and let $A \subset L_t^{unique}$ be a Borel set, where $t>0$. Then 
\begin{equation*}
\int_{M_t(A)} g(z)dz = \int_A g(M_t(\zeta))M_t'(\zeta)d\zeta,
\end{equation*}
where $M_t(\zeta):= \zeta^l(t)$. Moreover, for $\zeta \in L_t^{unique}$ 
\begin{equation*}
M_t'(\zeta) = e^{\int_0^t v(s)ds},
\end{equation*}
where $v(s)=u_x(s,\zeta(s))$.
\end{proposition}
\begin{remark}
By considering sets $A \cap L_t^{unique,N}$ instead of $A$ and passing to the limit $N\to \infty$, it suffices to assume that $g$ in Proposition \ref{Prop_ChangeOfV} is bounded on $L_t^{unique,N}$ for every $N=1,2,\dots$ with a bound possibly dependent on $N$.
\end{remark}
\begin{remark}
\label{Rem_0to0}
If $A \subset L_t^{unique,N}$ has measure $0$ then taking $g \equiv 1$ in Proposition \ref{Prop_ChangeOfV} we obtain $\mathcal{L}^1(M_t(A)) \le e^{Nt} \mathcal{L}^1(A) = 0$. Consequently, if $A \subset L_t^{unique}$ has measure $0$ then also $M_t(A)$ has measure $0$. A similar reasoning shows that if $A \subset L_t^{unique}$ satisfies $\mathcal{L}^1(M_t(A))=0$ then also $\mathcal{L}^1(A)=0$.
\end{remark}

Finally, let us define the notion of a pushforward of a set along characteristics.
\begin{definition}[Thick pushforward] 
\label{Def_Thick}
Let $u$ be a weak solution of \eqref{eq_WeakCH1}-\eqref{eq_WeakCH3}. For any Borel set $B \subset \mathbb{R}$ the \emph{thick pushforward} of $B$ from $0$ to $t$ with $t>0$ is defined as
\begin{equation*}
B(t):= \{\zeta(t): \zeta(\cdot) \mbox{ is any characteristic of $u$ satisfying } \zeta(0) \in B\}.
\end{equation*}
\end{definition}
\begin{remark}
If every characteristic of $u$ emanating from $B$ is unique forwards on $[0,t]$ then we can write simply $B(t) = \{\zeta(t): \zeta \in B\}$. Nevertheless, even in this case we will refer to $B(t)$ as thick pushforward of $B$.
\end{remark}

\section{Convergence lemmas for dissipative solutions}
\label{Sec_ConvLemmas}
In this section we formulate and prove two important lemmas, which will be instrumental in the proof of Theorem \ref{Th_STT}.
\begin{lemma} 
\label{Lem_limux2}
Let $u$ be a dissipative weak solution of \eqref{eq_WeakCH1}-\eqref{eq_WeakCH3}.
Let $B \subset \mathbb{R} \backslash \bigcup_{T>0}S_T$ be a bounded Borel set. Then 
\begin{equation*}
\lim_{t \to 0^+} \int_{B(t)} u_x^2(t,x) dx=0,
\end{equation*}
where $B(t)$ is the thick pushforward, see Definition \ref{Def_Thick}. 
\end{lemma}
\begin{proof}
Fix $\epsilon>0$. Let $D$ be so large that $B \subset [-D,D]$ and
\begin{equation*}
\int_{[-D,D]} (u_x^2(0,x) + u^2(0,x))dx \ge \int_{\mathbb{R}} (u_x^2(0,x)+u^2(0,x))dx -\epsilon.
\end{equation*}
Since Proposition \ref{Prop_MainCH} implies that $\bigcup_{T>0} S_T = \mathbb{R}$ and $S_{T_1} \supset S_{T_2}$ for $T_1 < T_2$, we can find 
 $T$  so small that 
\begin{equation*}
\int_{[-D,D] \cap S_T} (u_x^2(0,x)+u^2(0,x))dx> \int_{\mathbb{R}} (u_x^2(0,x)+u^2(0,x))dx-2\epsilon.
\end{equation*} 
Finally, recalling that $S_T = L_T^{unique} \backslash Z_T$ for some set $Z_T$ of measure $0$ and $L_{T}^{unique} = \bigcup_{N=1}^{\infty}L_{T}^{unique,N}$, where $\{L_T^{unique,N}\}_{N=1}^\infty$ is an increasing family of sets, take $N$ so large that 
\begin{equation*}
\int_{[-D,D] \cap S_T \cap L_T^{unique,N}} (u_x^2(0,x)+u^2(0,x))dx> \int_{\mathbb{R}} (u_x^2(0,x)+u^2(0,x))dx-3\epsilon.
\end{equation*} 
Denote $S_{D,T,N}:={[-D,D] \cap S_T \cap L_T^{unique,N}}$.
By Proposition \ref{Prop_ChangeOfV} and Proposition \ref{Prop_MainCH} we have for $\zeta \in S_T \cap L_T^{unique,N}$, $t \in [0,T]$ and $u(t) = u(t,\zeta(t))$, $v(t) = u_x(t,\zeta(t))$:
\begin{eqnarray}
M_t'(\zeta) &=& e^{\int_0^t v(s)ds}, \nonumber\\
e^{-tN} &\le& M_t'(\zeta) \le e^{tN},\nonumber\\
\frac{d}{dt} v^2M_t' &=& 2v\dot{v}M_t' + v^2 (vM_t') = vM_t'(2u^2-v^2 -2P+v^2) = 2vM_t'(u^2-P), \label{eq_ddtv2}\\
 \left| \frac{d}{dt} v^2 M_t'\right| &=& \left|2v(u^2-P)M_t' \right| \le 2N (\sup(u^2)+\sup(P)) e^{TN}=: C_1(N),\nonumber\\
\frac{d}{dt} u^2M_t' &=& 2u \dot{u} M_t' + u^2vM_t' = uM_t' (2\dot{u}+uv) = uM_t' (-2P_x + uv),\nonumber\\
\left|\frac{d}{dt} u^2 M_t'\right|&\le& \sup(|u|) e^{TN} (2\sup(|P_x|) + \sup(|u|)N)=:C_2(N).\nonumber
\end{eqnarray}
Let us recall that $u$, $P$, $P_x$ are bounded on $[0,\infty) \times \mathbb{R}$ due to the fact that $u \in L^{\infty}([0,\infty),H^1(\mathbb{R}))$. As a consequence,
$\sup(|u|), \sup(P), \sup(|P_x|)$ above are finite and so are $C_1(N),C_2(N)$.

Let $\tau \le \max\left(T, \frac {\epsilon} {2DC_1(N)},\frac {\epsilon} {2DC_2(N)}\right)$.
Then for $t \in [0,\tau]$ we have 
\begin{eqnarray*}
|v^2(t) M_t' - v^2(0)| &\le& \frac {\epsilon}{2D},\\
|u^2(t,\zeta(t)) M_t'(\zeta) - u^2(0,\zeta)| &\le& \frac {\epsilon}{2D}
\end{eqnarray*}
 and so, by the change of variables formula in Proposition \ref{Prop_ChangeOfV}, and denoting $S_{D,T,N}(t)$ the thick pushforward of $S_{D,T,N}$ (see Definition \ref{Def_Thick}), we obtain for $t \in [0,\tau]$
\begin{eqnarray*}
\int_{S_{D,T,N}(t)} u_x^2(t,x)dx &=& \int_{S_{D,T,N}} v^2(t)M_t'(\zeta) d\zeta\\ &\ge& \int_{S_{D,T,N}} \left(v^2(0) - \frac{\epsilon}{2D}\right) d\zeta\\ &\ge& \int_{S_{D,T,N}}u_x^2(0,\zeta)d\zeta - \epsilon, \\
\int_{S_{D,T,N}(t)} u^2(t,x)dx &=& \int_{S_{D,T,N}} u^2(t,\zeta(t))M_t'(\zeta) d\zeta\\ &\ge& \int_{S_{D,T,N}} \left(u^2(0,\zeta) - \frac{\epsilon}{2D}\right) d\zeta\\ &\ge& \int_{S_{D,T,N}}u^2(0,\zeta)d\zeta - \epsilon. \\
\end{eqnarray*}
Hence,
\begin{eqnarray*}
\int_{S_{D,T,N}(t)} (u_x^2(t,x)+u^2(t,x))dx &\ge& \int_{S_{D,T,N}} (u_x^2(0,\zeta)+u^2(0,\zeta))d\zeta - 2 \epsilon\\ &\ge& \int_{\mathbb{R}} (u_x^2(0,x)+u^2(0,x))dx - 5\epsilon.
\end{eqnarray*}

This means that for $t \in [0,\tau]$
\begin{equation*}
\int_{B(t)}u_x^2(t,x)dx \le 5\epsilon
\end{equation*}
due to the weak energy condition from Definition \ref{Def_dissipative} and the fact that $B(t) \cap S_{D,T,N}(t) = \emptyset$, which in turn is due to the uniqueness forwards and backwards of characteristics emanating from $S_{D,T,N}$. Due to arbitrariness of $\epsilon$ we conclude.
\end{proof}

\begin{lemma}
\label{Lem_L1Bt}
Let $u$ be a dissipative weak solution of \eqref{eq_WeakCH1}-\eqref{eq_WeakCH3}.
Let $B \subset \mathbb{R} \backslash \bigcup_{T>0}S_T$ be a bounded Borel set. Then 
\begin{equation}
\label{eq_limL1Bt}
\lim_{t \to 0^+} \mathcal{L}^1(B(t)) = 0,
\end{equation}
where $\mathcal{L}^1$ is the one-dimensional Lebesgue measure and $B(t)$ is the thick pushforward, see Definition \ref{Def_Thick}. 

\end{lemma}
\begin{proof}
Let $D>0$ be so large that $B \subset [-D,D]$. 
To prove \eqref{eq_limL1Bt} we first observe that $$B(t) \subset [(-D)^l(t),D^r(t)] \backslash S_t(t),$$ where $(-{D})^l(\cdot)$ and ${D}^r(\cdot)$ are the leftmost and rightmost characteristics emananating from $-D$ and $D$, respectively, and $S_t(t)$ is the thick pushforward, see Definition \ref{Def_Thick}. Consequently, $B(t)$ is bounded for $t>0$. Since $\mathcal{L}^1(\mathbb{R} \backslash\bigcup_{T>0}S_T)=0$ we have that for every $\epsilon>0$ there exists $T=T({\epsilon})>0$ such that 
$$\mathcal{L}^1([-{D},{D}]\backslash S_{T}) < \epsilon/8.$$
Recalling that $S_T = L_T^{unique} \backslash Z_T,$ where $Z_T$ has Lebesgue measure $0$ we can furthermore find $N>0$ so large that, denoting
$$S_{D,T,N}:= S_T \cap L_T^{unique,N} \cap [-{D},{D}],$$
we have
$$\mathcal{L}^1([-{D},{D}]\backslash S_{D,T,N}) < \epsilon/4$$
and hence $$\mathcal{L}^1(S_{D,T,N})>2D - \epsilon/4.$$
Finally, let us take any 
$$t \le \max \left(T,\frac \epsilon {4\sup(|u|)}, -\frac 1 N \log\left[1-\frac {\epsilon}{4\mathcal{L}^1(S_{D,T,N})}\right]\right).$$ 
Using Proposition \ref{Prop_ChangeOfV} we have
\begin{equation*}
\mathcal{L}^1(S_{D,T,N}(t)) = \int_{S_{D,T,N}} e^{\int_0^t u_x(s,\zeta(s))ds} d\zeta \ge \mathcal{L}^1(S_{D,T,N}) e^{-Nt} \ge \mathcal{L}^1(S_{D,T,N}) - \frac \epsilon 4.
\end{equation*}
Moreover, due to the $\sup(|u|)$ bound on the propagation speed of characteristics,
\begin{equation*}
\mathcal{L}^1([-D^l(t),{D}^r(t)])\le 2{D} + \frac \epsilon 2.
\end{equation*} 
 On the other hand, 
\begin{equation*}
B(t) \subset [(-{D})^l(t),{D}^r(t)] \backslash S_{D,T,N}(t).
\end{equation*}
Hence, 
\begin{eqnarray*}
\mathcal{L}^1(B(t)) &\le& \mathcal{L}^1([-{D}^l(t),{D}^r(t)] \backslash S_{D,T,N}(t))\\
&=& \mathcal{L}^1([-D^l(t),D^r(t)]) - \mathcal{L}^1(S_{D,T,N}(t)) \\
 &\le& 2{D} + \frac \epsilon 2 - \left(\mathcal{L}^1(S_{D,T,N}) - \frac \epsilon 4\right) \\ &=& 2{D} + \frac 3 4 \epsilon - \mathcal{L}^1(S_{D,T,N})\\ &\le& 2{D} + \frac {3}{4}\epsilon  - \left(2{D} - \frac \epsilon 4\right) = \epsilon.
\end{eqnarray*}
Since $\epsilon$ can be chosen arbitrarily small we conclude.
\end{proof}

\section{Backward solutions and characteristics}
\label{Sec_backward}
In this section we discuss the notion of backward solutions and backward characteristics as well as  their relation to the forward characteristics.
\begin{definition}
Let $T>0$ and let $u$ be a weak solution of \eqref{eq_WeakCH1}-\eqref{eq_WeakCH3}. The function 
\begin{equation*}
u^{Tb}(t,x):= -u(T-t,x)
\end{equation*}
is called a \emph{backward solution} of \eqref{eq_WeakCH1}-\eqref{eq_WeakCH3}.
\end{definition}
\begin{proposition}
If $u$ is a weak solution of \eqref{eq_WeakCH1}-\eqref{eq_WeakCH3} then the backward solution $u^{Tb}$ is a weak solution of \eqref{eq_WeakCH1}-\eqref{eq_WeakCH3} on $[0,T] \times \mathbb{R}$ with initial condition $u^{Tb}(0,x)=-u(T,x)$. 
\end{proposition}
\begin{proof}
Direct verification of Definition \ref{Def_Weak} 
\end{proof}
\begin{remark}
Dissipativity of $u$ does not imply dissipativity of $u^{Tb}$.
\end{remark}
\begin{lemma}
\label{Lem_corresp}
If $\gamma(\cdot)$ is a characteristic of $u^{Tb}$ satisfying 
\begin{itemize}
\item $\gamma(0) = \gamma$, 
\item $\frac d {dt}\gamma(t) = u^{Tb}(t,\gamma(t))$, 
\item $\frac d {dt} u^{Tb}(t,\gamma(t)) = -P_x^{Tb}(t,\gamma(t))$ 
\item $\gamma(T)=\zeta$
\end{itemize}
  then $\zeta(t):= \gamma(T-t)$ for $t \in [0,T]$ is a characteristic of $u$ such that $\gamma(\cdot)$ and $\zeta(\cdot)$ correspond to the same planar curve. Here, $P^{Tb} = \frac 1 2 e^{-|x|}* \left((u^{Tb})^2(t,\cdot) + \frac {(u^{Tb}_x)^2(t,\cdot)}{2}\right)$ is the functional $P$ computed in terms of $u^{Tb}$.

\end{lemma}
\begin{proof}
Indeed, $$\frac d {dt} \zeta(t) = -\gamma'(T-t) = -u^{Tb}(T-t,\gamma(T-t)) = -u^{Tb}(T-t,\zeta(t)) = u(t,\zeta(t))$$ and, similarly,
\begin{eqnarray*}
\frac {d}{dt} u(t,\zeta(t)) &=& \frac {d}{dt} u(t,\gamma(T-t)) = -\frac {d}{dt} u^{Tb}(T-t,\gamma(T-t))\\ &=& (u^{Tb}(\cdot,\gamma(\cdot)))' ({T-t})= -P_x^{Tb}(T-t,\gamma(T-t)) = -P_x(t,\zeta(t)).
\end{eqnarray*}
\end{proof}
Lemma \ref{Lem_corresp} implies that all characteristic curves on $[0,T] \times \mathbb{R}$ 
of $u^{Tb}$ correspond to characteristic curves of $u$ and vice-versa (note that $(u^{Tb})^{Tb} = u$ on $[0,T] \times \mathbb{R}$). Furthermore, observe that if $u$ is a dissipative solution satisfying $$u_x(t,x) \le const(1 + 1/t)$$ then 
\begin{equation}
\label{eq_TbDissip}
u^{Tb}_x(t,x) \ge -const(1 + 1/(T-t))
\end{equation}
for $0 \le t \le T$. 
Next, define 
\begin{equation}
\label{eq_omegagammakappa}
\omega^{\gamma,\kappa}(s):= \frac{u^{Tb}(s,\kappa(s))-u^{Tb}(s,\gamma(s))}{\kappa(s)-\gamma(s)},
\end{equation}
where $\gamma,\kappa \in \mathbb{R}$ and $\gamma(\cdot)$, $\kappa(\cdot)$ are certain characteristics of $u^{Tb}$ satisfying $\gamma(0)=\gamma$ and $\kappa(0)=\kappa$, respectively.
Finally, by $L_{Tb}^{unique,N}$ we will denote the set $L_T^{unique,N}$ corresponding to the backward solution $u^{Tb}$ and $L_{(T-1/K),Tb}^{unique,N}$ the set $L_{(T-1/K)}^{unique,N}$ corresponding to the backward solution $u^{Tb}$. More precisely, following Definition \ref{Def_Lt} we have the following definitions.

\begin{definition}
\label{def_LTKTb}
\begin{eqnarray}
L_{Tb}^{unique,N}&:=& \{\gamma \in \mathbb{R}: \gamma(\cdot) \mbox{ is unique forwards on $[0,T]$, }\nonumber \\ && \gamma(s) \mbox{ is a Lebesgue point of } u^{Tb}_x(s,\cdot) \mbox{ for almost every } s \in [0,T] \mbox{ and }\nonumber\\ &&\forall_{\kappa \in (\gamma- \frac 1 N, \gamma) \cup (\gamma,\gamma+ \frac 1 N),  s \in [0,T]} -N \le \omega^{\gamma,\kappa}(s) \le N \}, \nonumber \\ 
\label{eq_LTKTb}
L_{(T-1/K),Tb}^{unique,N}&:=& \{\gamma \in \mathbb{R}: \gamma(\cdot) \mbox{ is unique forwards on $[0,T-1/K]$, } \\ && \gamma(s) \mbox{ is a Lebesgue point of } u^{Tb}_x(s,\cdot) \mbox{ for almost every } s \in [0,T-1/K] \mbox{ and }\nonumber \\ &&\forall_{\kappa \in (\gamma- \frac 1 N, \gamma) \cup (\gamma,\gamma+ \frac 1 N),  s \in [0,T-1/K]} -N \le \omega^{\gamma,\kappa}(s) \le N \}, \nonumber\\
L_{Tb}^{unique} &:=& \bigcup_{N=1}^\infty L_{Tb}^{unique,N}, \nonumber\\
L_{(T-1/K),Tb}^{unique} &:=& \bigcup_{N=1}^\infty L_{(T-1/K),Tb}^{unique,N}. \nonumber
\end{eqnarray}
\end{definition}

\section{Representation of dissipative solutions in terms of characteristics}
\label{Sec_crucial}

In this section we prove the following theorem on dissipative solutions of \eqref{eq_WeakCH1}-\eqref{eq_WeakCH3}, which is the main technical result of this paper. It states that characteristics, along which $u_x$ evolves, fill up, up to a set of measure $0$, the whole strip $[0,T] \times \mathbb{R}$ for $T$ small enough.
\begin{theorem}
\label{Th_STT}
Let $u$ be a dissipative weak solution of \eqref{eq_WeakCH1}-\eqref{eq_WeakCH3}. 
Suppose that the sets $S_T$ in Proposition \ref{Prop_MainCH} have the form $S_T = L_T^{unique} \backslash Z_T$, where $Z_T$ has Lebesgue measure $0$. Then $S_T(T)$ is a set of full Lebesgue measure for every $T>0$ small enough, where $S_T(T) = \{\zeta(T): \zeta \in S_T \mbox{ and }\zeta(\cdot) \mbox{ is the unique characteristic satisfying } \zeta(0)=\zeta \}.$
\end{theorem}

To begin the proof, let us fix $T>0$ which satisfies $T<T_{max}$, where 
\begin{equation}
\label{eq_defTmax}
T_{max}:=\frac {\pi}{8 \sqrt{LC}},
\end{equation}
$L$ is a fixed constant, which can be chosen $1$ for the Camassa-Holm equation, see Lemma \ref{LemA45} below, and 
\begin{equation}
\label{eq_defC}
C:=\sup_{t \in [0,\infty)} \int_{\mathbb{R}}\left[u^2(t,y)+ \frac 1 2 u_x^2(t,y)\right]dy.
\end{equation}
 For $u$ -- dissipative, $C$ can be bounded, by the weak energy condition in Definition \ref{Def_dissipative}, in terms of the initial condition:
\begin{equation*}
C \le \sup_{t \in [0,\infty)} \int_{\mathbb{R}}[u^2(t,y)+ u_x^2(t,y)]dy = \int_{\mathbb{R}}[u^2(0,y)+ u_x^2(0,y)]dy.
\end{equation*}

We split the remaining part of the proof of Theorem \ref{Th_STT} into two parts:
\begin{itemize}
\item In Part 1 we consider the set $\Gamma \ni \gamma$ of all $\gamma$ such that $u^{Tb}$ has, along characteristics emanating from $\gamma$, difference quotients $\omega^{\gamma,\kappa}$ (defined by \eqref{eq_omegagammakappa}) bounded from below. We show that for almost all $\gamma \in \Gamma$ we have $\gamma(T) \in S_T$.
\item In Part 2 we consider all the remaining starting points of characteristics of $u^{Tb}$, i.e. the set $\mathbb{R} \backslash \Gamma$, and show that this set has Lebesgue measure null.
\end{itemize}

\noindent {\bf Part 1.}

Define the following sets 
\begin{eqnarray*}
\Gamma^K &:=& \{\gamma \in \mathbb{R}: u^{Tb}_x(0,\gamma)=\lim_{h\to 0} \frac{u^{Tb}(0,\gamma+h)-u^{Tb}(0,\gamma)}{h} \mbox{ and }\\&&\mbox{ for every characteristic } \gamma(\cdot) \mbox{ of } u^{Tb} \mbox{ with } \gamma(0)=\gamma \mbox{ and  }  \\
&&\mbox{ for every } \kappa \in (\gamma-1/K, \gamma) \cup (\gamma,\gamma + 1/K) \mbox{ and every characteristic } \kappa(\cdot) \mbox{ of } u^{Tb}\\
&&\mbox{ satisfying } \kappa(0)=\kappa \mbox{ and every } s \in [0,T] \mbox{ we have } \omega^{\gamma,\kappa}(s) \ge -K\},\\
\Gamma &:=& \bigcup_{K=1}^{\infty} \Gamma^K,
\end{eqnarray*}

\begin{proposition}
\label{Prop57new}
\begin{equation}
\label{eqLem57new}
\Gamma^K = \bigcup_{N \in \mathbb{N}} \Gamma^{K,unique,N} \cup \tilde{Z}_K,
\end{equation}
where $\tilde{Z}_K$ is a set of measure $0$ and
\begin{eqnarray*}
\Gamma^{K,unique,N} &:=& \{\gamma \in \mathbb{R}: u^{Tb}_x(0,\gamma)=\lim_{h\to 0} \frac{u^{Tb}(0,\gamma+h)-u^{Tb}(0,\gamma)}{h} \mbox{ and }\\&&\mbox{for every characteristic } \gamma(\cdot) \mbox{ of } u^{Tb} \mbox{ with } \gamma(0)=\gamma \mbox{ and every } \\ &&\kappa \in (\gamma-1/K, \gamma) \cup (\gamma,\gamma + 1/K) \mbox{ and every characteristic } \kappa(\cdot) \mbox{ of } u^{Tb}\\
&&\mbox{satisfying } \kappa(0)=\kappa \mbox{ and every } s \in [0,T] \mbox{ we have }  
 N \ge \omega^{\gamma,\kappa}(s) \ge -K\}. 
\end{eqnarray*}

\end{proposition}
The proof of Proposition \ref{Prop57new} is a modification of the proof of \cite[Lemma 5.7]{CHGJ}. Since, however, Proposition \ref{Prop57new} does not follow directly from   \cite[Lemma 5.7]{CHGJ} due to different assumptions, we provide a proof here for the sake of completeness. Let us begin with two useful lemmas.
\begin{lemma}
\label{Lem_uniquebackforw}
If $\gamma \in \Gamma^K$ then $\gamma(\cdot)$ is unique backwards on $[0,T]$, i.e. $\gamma(\cdot)$ does not cross with any other characteristic on $[0,T]$. If, moreover, $\gamma \in \Gamma^{K,unique,N}$ then $\gamma(\cdot)$ is also unique forwards on $[0,T]$.
\end{lemma}

\begin{proof}
Suppose that $\gamma \in \Gamma^K$ and 
take $\kappa \in (\gamma, \gamma + 1/K)$. Then  $$\frac {d}{dt}(\kappa(t)-{\gamma}(t)) = u^{Tb}(t,\kappa(t)) - u^{Tb}(t,\gamma(t))\ge -K(\kappa(t)-{\gamma}(t))$$
and hence $\kappa(t) - \gamma(t) \ge e^{-Kt} (\kappa-\gamma)$. Consequently, $\kappa(t) \neq \gamma(t)$ for $t \in [0,T]$. If $\kappa \ge \gamma + 1/K$ then any characteristic emanating from $\gamma + 1/K$ and crossing $\gamma(\cdot)$ on $[0,T]$ would have to cross a characteristic emanating form $\gamma + 1/2K$ for some $\tau \in [0,T]$. Concatenation of the two would produce a characteristic emanating from $\gamma + 1/2K$ which crosses $\gamma(\cdot)$ at time $\tau$, which is a contradiction. A similar reasoning can be conducted for $\kappa<\gamma$. This proves uniqueness backwards for $\gamma \in \Gamma^K$.
 
To prove uniqueness forwards, assume now that $\gamma \in \Gamma^{K,unique,N}$ and suppose that for some $s_0 \in [0,T]$ we have $\gamma^l(s_0)<\gamma^r(s_0)$, where $\gamma^l$ and $\gamma^r$ are the leftmost and rightmost characteristics of $u^{Tb}$, respectively. 
For $\kappa \in (\gamma,\gamma+1/K)$ we have $$\frac {d}{dt}(\kappa(t)-{\gamma}^l(t)) = u^{Tb}(t,\kappa(t)) - u^{Tb}(t,\gamma^l(t))\le N(\kappa(t)-{\gamma}^l(t)).$$ Consequently, using uniqueness backwards of $\gamma^r(\cdot)$, we obtain $0< \gamma^r(s_0) - \gamma^l(s_0)< \kappa(s_0)-\gamma^l(s_0) \le e^{s_0 N} (\kappa - \gamma)$. Passing to the limit $\kappa \to \gamma^+$ we obtain a contradiction.
\end{proof}
\begin{remark}
\label{Rem_44}
If $u$ is dissipative then for every $\gamma \in \mathbb{R}$ every characteristic $\gamma(\cdot)$ of $u^{Tb}$ is unique backwards on $[0,T-\epsilon]$ for any $\epsilon \in (0,T)$.  Indeed, using \eqref{eq_TbDissip} and reasoning as in Lemma \ref{Lem_uniquebackforw} we obtain 
\begin{eqnarray*}
\frac d {dt} (\kappa(t)-\gamma(t)) &=& u^{Tb}(t,\kappa(t)) - u^{Tb}(t,\gamma(t)) = \int_{\gamma(t)}^{\kappa(t)} u^{Tb}_x(t,y)dy \\
&\ge& -const(1+1/(T-t)) (\kappa(t)-\gamma(t)) \ge -const (1+1/\epsilon) (\kappa(t)-\gamma(t)) 
\end{eqnarray*}
which implies $\kappa(t) - \gamma(t) \ge (\kappa - \gamma)e^{-const (1+1/\epsilon)t}$.
Note that this reasoning does not exclude the collision of characteristics $\gamma(\cdot)$ and $\kappa(\cdot)$ at time $T$, which indeed does not occur either and will be proven later as a conclusion from Part 2 of the proof of Theorem \ref{Th_STT}, see Lemma \ref{Lem_charunique}.
\end{remark}
\begin{lemma}
\label{Lem_forwards}
For almost every $\gamma \in \Gamma$ the characteristic $\gamma(\cdot)$ is unique forwards. 
\end{lemma}
\begin{proof}
Consider the sets $\Xi^\gamma = \{(t,y): t \in [0,T], y \in [\gamma^l(t),\gamma^r(t)]\}$. If $\gamma(\cdot)$ is not unique forwards then $\Xi^\gamma$ has positive two-dimensional Lebesgue measure. Since the sets $\Xi^{\gamma_1}$ and $\Xi^{\gamma_2}$ are, by Lemma \ref{Lem_uniquebackforw}, disjoint for $\gamma_1 \neq \gamma_2$, $\gamma_1,\gamma_2 \in \Gamma$ we conclude that only countably many of sets $\Xi^\gamma$ can have positive two-dimensional Lebesgue measure. Consequently, for all but at most countably many $\gamma \in \Gamma$ the characteristic $\gamma(\cdot)$ is unique forwards.
\end{proof}

\begin{proof}[Proof of Proposition \ref{Prop57new}]
To begin, we can assume, by Lemma \ref{Lem_forwards}, that, without loss of generality, for every $\gamma \in \Gamma^K$ the characteristic $\gamma(\cdot)$ is not only unique backwards but also unique forwards. Next, fix $K$ and define 
\begin{eqnarray*}
J_T^{bad}&:=& \{\gamma \in \Gamma^K: \forall_{\delta>0} \forall_{M>0} \exists_{\kappa \in (\gamma,\gamma+\delta)} \exists_{\kappa(\cdot)} \exists_{s \in [0,T]} \omega^{\gamma,\kappa}(s)>M\},\\
\tilde{J}_T^{bad}&:=& \{\gamma \in \Gamma^K: \forall_{\delta>0} \forall_{M>0} \exists_{\kappa \in (\gamma-\delta,\gamma)} \exists_{\kappa(\cdot)} \exists_{s \in [0,T]} \omega^{\gamma,\kappa}(s)>M\},
\end{eqnarray*}
where $\omega^{\gamma,\kappa}$ is given by \eqref{eq_omegagammakappa}  
and $\kappa(\cdot)$ is some characteristic of $u^{Tb}$ emanating from $\kappa$. Our goal is  to show that $\mathcal{L}^1(J_T^{bad})=\mathcal{L}^1(\tilde{J}_T^{bad})=0$, which directly implies \eqref{eqLem57new}.
To prove this, define for every $\gamma \in J_T^{bad}$ and every $M,\delta$ with $\delta<\frac 1 K$,
\begin{equation*}
\Pi_{\gamma}^{M,\delta}:= [\gamma,\kappa],
\end{equation*}
where $\kappa$ is any number such that $\kappa \in (\gamma,\gamma+\delta)$ and there exists a characteristic $\kappa(\cdot)$ of $u^{Tb}$, satisfying $\kappa(0)=\kappa$ and $s \in [0,T]$ such that $\omega^{\gamma,\kappa}(s)>M$. Having fixed the sets $\Pi_\gamma^{M,\delta}$ we note that for every fixed $M$ the family
 $$\mathcal{E}^M := \{\Pi_{\gamma}^{M,\delta}, \delta>0, \gamma \in J_T^{bad}\}$$ is a covering of $J_T^{bad}$ by closed intervals. By the Vitali covering theorem we obtain an at most countable pairwise disjoint family of closed intervals $\mathcal{F}^M \subset \mathcal{E}^M$ such that 
\begin{equation*}
J_T^{bad} \subset \bigcup \mathcal{F}^M
\end{equation*}
holds up to a set of measure $0$. The family $\mathcal{F}^M$ can be represented as $\{[\gamma_i,\kappa_i]\}_{i=1}^{\infty}$ with fixed characteristics $\kappa_i(\cdot)$ and fixed timepoints $s_i$ satisfying $\kappa_i(0)=\kappa_i$ and $\omega^{\gamma_i,\kappa_i}(s_i)>M$. 

Now, denoting $h_i(s):=\kappa_i(s)-\gamma_i(s)$  and $\omega_i(s):=\omega^{\gamma_i,\kappa_i}(s)$ and observing that $\frac d {ds} h_i(s) = \omega_i(s)h_i(s)$ we obtain
\begin{equation}
\label{eq_571}
h_i(s_i) = h_i(0)e^{\int_{0}^{s_i} \omega_i(s)ds} \ge h_i(0)e^{-KT}.
\end{equation}
Let us now prove some estimates for $\omega_i$, following \cite[Lemma 5.4]{CHGJ}.
\begin{lemma}
\label{LemA45}
Denote $A(x,y):= \frac 1 2 sgn(x-y) e^{-|x-y|}$. 
Function $\omega_i$ satisfies for $0 \le t \le T$ and $0 \le t_1 \le t_2\le T$
\begin{eqnarray}
\dot{\omega_i}(t) &\ge& -\omega_i^2(t) - LC,\label{eq_lem451}\\
\omega_i(t_2) &\ge& \sqrt{LC} \tan \left(-\sqrt{LC}(t_2-t_1) + \arctan\left(\frac {\omega_i(t_1)}{\sqrt{LC}} \right)\right)\label{eq_lem452}
\end{eqnarray}
where $L \in (0,\infty)$ is a constant such that  $\frac {A(x_2,y) -A(x_1,y)}{x_2 - x_1} \ge -L$ for every $x_2>x_1$ and every $y \in \mathbb{R}$ and $C$ is defined by \eqref{eq_defC}.  
\end{lemma}
\begin{proof}
First observe that for $x_2>x_1$ we have 
\begin{eqnarray*}
A(x_2,y) - A(x_1,y) &=& \frac 1 2 (sgn(x_2-y) - sgn(x_1-y))e^{-|x_2-y|} \\ &&+ \frac 1 2 sgn(x_1-y)(e^{-|x_2-y|} - e^{-|x_1-y|})\\ &\ge& - \frac 1 2 (x_2 - x_1), 
\end{eqnarray*}
where we used the nonnegativity of $\frac 1 2 (sgn(x_2-y) - sgn(x_1-y))e^{-|x_2-y|}$ and Lipschitz continuity of $e^{-|x|}$ with constant $1$. Now, the Camassa-Holm equation can be written as
\begin{equation*}
u_t + (u^2/2)_x = \int_\mathbb{R} A(x,y)\left[u^2(t,y)+ \frac 1 2 u_x^2(t,y)\right]dy.
\end{equation*}
Denoting $p_i := \omega_i h_i = u^{Tb}(s,\kappa_i(s)) - u^{Tb}(s,\gamma_i(s))$, we thus obtain
\begin{eqnarray*}
\dot{\omega_i} = \frac {\dot{p_i}}{h_i} - \frac {p_i \dot{h_i}}{h_i^2} &=& - \omega_i^2 + \frac 1 {h_i} \int_\mathbb{R} (A(\kappa_i(s),y) - A(\gamma_i(s),y)) \left[(u^{Tb})^2(t,y) + \frac 1 2 (u^{Tb}_x)^2(t,y)\right]dy \\ &\ge & -\omega_i^2 - L \int_\mathbb{R} ((u^{Tb})^2(t,y) + \frac 1 2 (u^{Tb}_x)^2(t,y)) dy \ge -\omega_i^2 - LC,
\end{eqnarray*}
where we used the fact that $$\int_\mathbb{R} \left((u^{Tb})^2(t,y) + \frac 1 2 (u^{Tb}_x)^2(t,y)\right) dy = \int_{\mathbb{R}} \left(u^2(T-t,y) + \frac 1 2 u_x^2(T-t,y)\right) dy \le C.$$ This proves \eqref{eq_lem451}. To prove \eqref{eq_lem452} we observe that $\Omega_i(t):=\sqrt{LC} \tan(-\sqrt{LC}(t-t_1) + \arctan\left(\frac {\omega_i(t_1)}{\sqrt{LC}} \right)$ solves the equation $\dot{\Omega}_i(t) = - \Omega_i^2(t) - LC$ with initial condition $\Omega_i(t_1) = \omega_i(t_1)$. Thus, $\frac d {dt} (\Omega_i - \omega_i) \le -\Omega_i^2 + \omega_i^2 = - (\Omega_i + \omega_i)(\Omega_i - \omega_i)$ and hence, by Gronwall's inequality, $\Omega_i - \omega_i \le 0$. 
\end{proof}

Returning to the proof of Proposition \ref{Prop57new} let $M \ge \sqrt{LC} \tan (3\pi/8)$. Then $\omega_i(s_i)>M$ implies 
\begin{equation}
\label{eq_5705}
\omega_i(t)\ge \sqrt{LC}
\end{equation} 
for every $t \in (s_i,T]$, where we used \eqref{eq_lem452} and the assumption $T<T_{max}$, see \eqref{eq_defTmax}. Then, using \eqref{eq_lem451} and equality $\dot{h}_i = p_i = h_i \omega_i$ we obtain 
\begin{equation*}
\frac d {dt} (h_i \omega_i) = \dot{h_i}\omega_i+h_i \dot{\omega_i} \ge \omega_i^2 h_i +h_i(-\omega_i^2-LC)  = - h_i LC = - (h_i \omega_i) \frac {LC}{\omega_i} \ge -\sqrt{LC}(h_i\omega_i).
\end{equation*}
Hence, 
\begin{equation}
\label{eq_572}
h_i(T)\omega_i(T) \ge e^{-(T-s_i)\sqrt{LC}} h_i(s_i)\omega_i(s_i) \ge e^{-T\sqrt{LC}} h_i(s_i)\omega_i(s_i).
\end{equation}
Using \eqref{eq_571}-\eqref{eq_572}, and the fact that characteristics $\gamma_i(\cdot)$ do not collide with any other characteristics on $[0,T]$ and hence $[\gamma_i(t),\kappa_i(t)]$ are disjoint and of positive length for every $t \in [0,T]$, we calculate
\begin{eqnarray*}
\int_{\mathbb{R}} (u^{Tb}_x)^2(T,x)dx &\ge& \sum_{i=1}^{\infty} \int_{\gamma_i(T)}^{\kappa_i(T)} (u_x^{Tb})^2(T,x)dx \\
&\ge& \sum_{i=1}^{\infty} (\kappa_i(T)-\gamma_i(T)) \frac {1}{\kappa_i(T)-\gamma_i(T)} \int_{\gamma_i(T)}^{\kappa_i(T)} (u_x^{Tb})^2(T,x)dx \\
&\ge& \sum_{i=1}^{\infty} (\kappa_i(T)-\gamma_i(T)) \left[ \frac {1}{\kappa_i(T)-\gamma_i(T)} \int_{\gamma_i(T)}^{\kappa_i(T)} u_x^{Tb}(T,x)dx \right]^2 \\
&\ge& \sum_{i=1}^{\infty} h_i(T)\omega_i^2(T)\\&\ge& \sqrt{LC}\sum_{i=1}^{\infty} h_i(T) \omega_i(T)\\
&\ge& \sqrt{LC}\sum_{i=1}^{\infty} h_i(s_i)\omega_i(s_i)e^{-T\sqrt{LC}} \\
&\ge& M\sqrt{LC}\sum_{i=1}^{\infty} h_i(s_i)e^{-T\sqrt{LC}} \\
&\ge& M\sqrt{LC}\sum_{i=1}^{\infty} h_i(0)e^{-KT}e^{-T\sqrt{LC}} \\
&\ge& M\sqrt{LC}e^{-KT}e^{-T\sqrt{LC}} \mathcal{L}^1(J_T^{bad}). 
\end{eqnarray*}
Passing to the limit $M\to \infty$ we conclude that $\mathcal{L}^1(J_T^{bad})=0$. The proof for $\tilde{J}_T^{bad}$ is similar.
\end{proof}

Continuing Part 1 of the proof of Theorem \ref{Th_STT}, for every $K=1,2,\dots$, we obtain by Proposition \ref{Prop57new} that
\begin{equation*}
\Gamma^K = \bigcup_{N \in \mathbb{N}} \Gamma^{K,unique,N} \cup \tilde{Z}_K,
\end{equation*}
where $\tilde{Z}_K$ are sets of measure $0$. Moreover, observe that if $K \le N$ then $\Gamma^{K,unique,N} \subset \Gamma^{N,unique,N}$
and hence

\begin{equation}
\label{eq_GammaDecomp}
\Gamma = \bigcup_{K=1}^\infty \Gamma^K = \bigcup_{K=1}^\infty \left(\bigcup_{N=1}^\infty \Gamma^{K,unique,N}\right)\cup \tilde{Z}_K = \bigcup_{N=1}^\infty \Gamma^{N,unique,N} \cup \tilde{Z}_N. 
\end{equation}

\begin{lemma}
\label{Lem_5757}
\begin{enumerate}[i)]
\item There exist sets $\hat{Z}_N \subset \Gamma^{N,unique,N}$ such that $\mathcal{L}^1(\hat{Z}_N)=0$ and for every $\gamma \in \Gamma^{N,unique,N} \backslash \hat{Z}_N$ we have that $\gamma(T)$ is a Lebesgue point of $u_x^{Tb}(T,\cdot)$ and $\gamma(t)$ is a Lebesgue point of $u_x^{Tb}(t,\cdot)$ for almost every $t \in [0,T]$.
\item $\Gamma^{N,unique,N} \backslash \hat{Z}_N \subset L_{Tb}^{unique,N} $.
\item For every $\zeta \in ((\gamma-1/N)^r(T),(\gamma+1/N)^l(T))$ we have, for $0\le t \le T$, that  $$\zeta(t) \in [(\gamma-1/N)^r(T-t) , (\gamma+1/N)^l(T-t)],$$ where $\zeta(\cdot)$ is any characteristic of $u$ satisfying $\zeta(0) = \zeta$ and $(\gamma+1/N)^l(\cdot)$, $(\gamma-1/N)^r(\cdot)$ are the leftmost and rightmost characteristics of $u^{Tb}$ satisfying $(\gamma+1/N)^l(0)=\gamma + 1/N$ and $(\gamma-1/N)^r(0)=\gamma - 1/N$, respectively.
\item For every $\gamma \in \Gamma^{N,unique,N}\backslash \hat{Z}_N$ we have $\gamma(T) \in L_T^{unique,\tilde{N}}$ for $\tilde{N}$ large enough. 
\item There exists  a set $\hat{\hat{Z}}_N$ such that $\mathcal{L}^1(\hat{\hat{Z}}_N) = 0$ and if $\gamma \in \Gamma^{N,unique,N} \backslash (\hat{Z}_N \cup \hat{\hat{Z}}_N) $ then $\gamma(T) \in L_T^{unique,\tilde{N}} \backslash Z_T$ for some $\tilde{N} \ge N$, where $Z_T$ is the null set from Proposition \ref{Prop_MainCH}. 
\end{enumerate}
\end{lemma}
\begin{proof}
\begin{enumerate}[i)]
\item First observe that for every $t \in [0,T]$ the mapping $M_t: \Gamma^{N,unique,N}\to \mathbb{R}$ defined by $M_t(\gamma) = \gamma(t)$, where $\gamma(\cdot)$ is the unique characteristic of $u^{Tb}$ satisfying $\gamma(0)=\gamma$, is bijective, Lipschitz continuous, and its inverse is Lipschitz continuous. Indeed, bijectivity follows by Lemma \ref{Lem_uniquebackforw} and Lipschitz continuity of $M_t$ and its inverse follows from estimate $-N \le \omega^{\gamma,\kappa} \le N$ whenever $\kappa \in (\gamma-1/N,\gamma) \cup (\gamma,\gamma+1/N)$, which implies $e^{-tN}(\kappa - \gamma) \le M_t(\kappa) - M_t(\gamma) \le e^{tN}(\kappa - \gamma)$. Consequently, $M_t$ and $M_t^{-1}$ map null sets to null sets. Now, the set of Lebesgue points of $u_x^{Tb}(T,\cdot)$ is a set of full Lebesgue measure and so is its image under mapping $M_t^{-1}$. Similarly, the set of Lebesgue points of $u_x^{Tb}(t,\cdot)$ has full Lebesgue measure for every $t \in [0,T]$. Hence, the set $\{(t,y): y=M_t(\gamma) \mbox{ for some } \gamma \in \Gamma^{N,unique,N}  \mbox{ and } y \mbox{ is not a Lebesgue point of } u_x^{Tb}(t,\cdot) \}$ has, by Fubini theorem, two-dimensional Lebesgue measure $0$.  Since the mapping $(t,\gamma) \mapsto (t,M_t(\gamma))$ is Lipschitz continuous with Lipschitz inverse on $[0,T] \times \Gamma^{N,unique,N}$ we conclude that $\{(t,\gamma): M_t(\gamma) \mbox{ is not a Lebesgue point of } u_x^{Tb}(t,\cdot)\}$ has also two-dimensional measure $0$ and so, again by the Fubini theorem, for almost every $\gamma$ the set of $t$ such that $M_t(\gamma)$ is not a Lebesgue point of $u_x^{Tb}(t,\cdot)$ has measure $0$. 
\item Follows from i) and definitions of $\Gamma^{N,unique,N}$ and $L_{Tb}^{unique,N}$, see formulation of Proposition \ref{Prop57new} and Definition \ref{def_LTKTb}.

\item Suppose $\zeta(\tau) = (\gamma+1/N)^l(T-\tau)$ for some $\tau \in (0,T]$. Let $\tau$ be the smallest one satisfying this property, which exists due to continuity of characteristics and inequality $\zeta<(\gamma+1/N)^l(T)$. Then the characteristic ${(\gamma+1/N)}(\cdot)$ of $u^{Tb}$ defined by $$(\gamma+1/N)(s):= \begin{cases} 
(\gamma+1/N)^l(s) &\mbox { if } s \in [0,T-\tau],\\
\zeta(T-s) &\mbox { if } s \in (T-\tau,T]
\end{cases}$$
satisfies $(\gamma+1/N)(s) \le (\gamma+1/N)^l(s)$ for every $s \in [0,T]$ and $(\gamma+1/N)(T) < (\gamma+1/N)^l(T),$ which contradicts the definition of leftmost characteristic.
We argue similarly when $\zeta(\tau) = (\gamma-1/N)^r(T-\tau)$ for some $\tau \in (0,T]$.

\item Take  $\tilde{N}$ so large that $\tilde{N} \ge N$ and  $$[\gamma(T)-1/\tilde{N},\gamma(T)+1/\tilde{N}] \subset ((\gamma-1/N)^r(T),(\gamma+1/N)^l(T)).$$ Using iii) we obtain that for every $\eta \in [\gamma(T)-1/\tilde{N},\gamma(T)) \cup (\gamma(T),\gamma(T)+1/\tilde{N}]$ and every characteristic $\eta(\cdot)$ of $u$ such that $\eta(0)=\eta$  there exists $\kappa \in (\gamma-1/N,\gamma+1/N)$ and a characteristic $\kappa(\cdot)$ of $u^{Tb}$ such that $\kappa(0)=\kappa$ and  $\eta(t) = \kappa(T-t)$ for every $t \in [0,T]$. Indeed, it is enough to take $\kappa = \eta(T)$ and $\kappa(\cdot)$ defined by $\kappa(t) := \eta(T-t)$. Note that $\kappa(\cdot)$ is indeed a characteristic of $u^{Tb}$ by Lemma \ref{Lem_corresp}. We conclude using ii), the fact that, denoting $\zeta=\gamma(T)$, we have
\begin{eqnarray*}
\omega(t) &=& \frac {u(t,\eta(t)) - u(t,\zeta(t))}{\eta(t)-\zeta(t)} = \frac {u(t,\kappa(T-t)) - u(t,\gamma(T-t))}{\kappa(T-t)-\gamma(T-t)}\\ &=& \frac {-u^{Tb}(T-t,\kappa(T-t)) +  u(T-t,\gamma(T-t))}{\kappa(T-t)-\gamma(T-t)} = -\omega^{\gamma,\kappa}(T-t)
\end{eqnarray*}
and hence $-\tilde{N} \le \omega(t) \le \tilde{N}$ iff $-\tilde{N} \le \omega^{\gamma,\kappa}(T-t) \le \tilde{N}$ 
 and verifying the definition of $L_T^{unique,\tilde{N}}$.
\item Follows by the fact that $\mathcal{L}^1(Z_T) = 0$ and Remark \ref{Rem_0to0}. 
\end{enumerate}
\end{proof}

From Lemma \ref{Lem_5757} we conclude that  if $\gamma \in \bigcup_{N=1}^\infty \Gamma^{N,unique,N} \backslash (\hat{Z}_N \cup \hat{\hat{Z}}_N)$ then $$\gamma(T) \in \bigcup_{\tilde{N}} L_T^{unique,\tilde{N}} \backslash Z_T = L_T^{unique} \backslash Z_T = S_T$$ and hence $\gamma \in S_T(T)$, where $S_T$ is given in Proposition \ref{Prop_MainCH} and $S_T(T)=\{\gamma(T): \gamma \in S_T\}$ is the thick pushforward. Thus,
\begin{equation*}
\bigcup_{N=1}^\infty \Gamma^{N,unique,N} \backslash (\hat{Z}_N \cup \hat{\hat{Z}}_N) \subset S_T(T)
\end{equation*}
and therefore, by \eqref{eq_GammaDecomp}, there exists a null set $Z^\Gamma \subset \Gamma$ such that 
\begin{equation*}
\Gamma \backslash Z^\Gamma \subset S_T(T).
\end{equation*}

\noindent {\bf Part 2.}
Let us consider now $\gamma \in \mathbb{R} \backslash \Gamma$. For every $K=1,2,\dots$ define  

\begin{eqnarray*}
A^K &:=& \{\gamma \in \mathbb{R}\backslash \Gamma: u^{Tb}_x(0,\gamma)=\lim_{h\to 0} \frac{u^{Tb}(0,\gamma+h)-u^{Tb}(0,\gamma)}{h} \mbox{ and }\\&&\mbox{ for every characteristic } \gamma(\cdot) \mbox{ of } u^{Tb} \mbox{ with } \gamma(0)=\gamma \mbox{ and  }  \\
&&\mbox{ for every } \kappa \in (\gamma-1/K, \gamma) \cup (\gamma,\gamma + 1/K) \mbox{ and every characteristic } \kappa(\cdot) \mbox{ of } u^{Tb}\\
&&\mbox{ satisfying } \kappa(0)=\kappa \mbox{ and every } s \in [0,T-1/K] \mbox{ we have } \omega^{\gamma,\kappa}(s) \ge - const (1+K)\},
\end{eqnarray*}
where $\omega^{\gamma,\kappa}$ was defined in \eqref{eq_omegagammakappa} and $const$ is the same as in \eqref{eq_TbDissip}.

Note that $A^K$ is a full measure subset of $\mathbb{R} \backslash \Gamma$ by \eqref{eq_TbDissip}.
Reasoning exactly as in Part 1 for $\Gamma^K$ (the only difference being that in Part 1 we had a time interval $[0,T]$ and the lower bound $\omega^{\gamma,\kappa} \ge -K$, whereas here we have the time interval $[0,T-1/K]$ and lower bound $\omega^{\gamma,\kappa} \ge -const(1+K)$) we obtain 
\begin{equation*}
A^K = \bigcup_{N \in \mathbb{N}} A^{K,unique,N} \cup \tilde{Z}_K,
\end{equation*}
where $\mathcal{L}^1(\tilde{Z}_K)=0$ and 
\begin{eqnarray*}
A^{K,unique,N} &:=& \{\gamma \in \mathbb{R}\backslash \Gamma: u^{Tb}_x(0,\gamma)=\lim_{h\to 0} \frac{u^{Tb}(0,\gamma+h)-u^{Tb}(0,\gamma)}{h} \mbox{ and }\\&&\mbox{ for every characteristic } \gamma(\cdot) \mbox{ of } u^{Tb} \mbox{ with } \gamma(0)=\gamma \mbox{ and  }  \\
&&\mbox{ for every } \kappa \in (\gamma-1/K, \gamma) \cup (\gamma,\gamma + 1/K) \mbox{ and every characteristic } \\ &&\kappa(\cdot) \mbox{ of } u^{Tb} \mbox{ satisfying } \kappa(0)=\kappa \mbox{ and every } s \in [0,T-1/K] \\ &&\mbox{we have } N \ge\omega^{\gamma,\kappa}(s) \ge - const (1+K)\}.
\end{eqnarray*}

Moreover, as in Lemma \ref{Lem_5757}i,ii we obtain that, for every $N> const(1+K)$, there exists a Lebesgue null set $\hat{Z}_{K,N}$ 
such that 
\begin{equation*}
A^{K,unique,N} \backslash \hat{Z}_{K,N} \subset L_{(T-1/K),Tb}^{unique,N}
\end{equation*}
and $L_{(T-1/K),Tb}^{unique,N}$ is defined by \eqref{eq_LTKTb}.

This, by Proposition \ref{Prop_MainCH} applied to the backward solution $u^{Tb}$, implies that for almost every $\gamma \in A^{K,unique,N}$, where $N>const(1+K)$, the characteristic of $u^{Tb}$ emanating from $\gamma$ is unique, satisfies
\begin{equation}
\label{eq_KN}
\forall_{\kappa \in (\gamma- \frac 1 N, \gamma) \cup (\gamma,\gamma + \frac 1 N),\kappa(\cdot), s \in [0,T-1/K]} N \ge \omega^{\gamma,\kappa} \ge -N
\end{equation}
and satisfies equation \eqref{Eq_PropL}, with $u^{Tb}$ and $P^{Tb}$ in place of $u$ and $P$, respectively, i.e. $\gamma(\cdot)$ satisfies
\begin{equation}
\label{Eq_PropLTb}
\frac {d}{dt} u^{Tb}_x(t,\gamma(t)) = (u^{Tb})^2(t,\gamma(t)) - \frac 1 2 (u_x^{Tb})^2(t,\gamma(t)) - P^{Tb}(t,\gamma(t))
\end{equation}
for $t \in [0,T-1/K]$, where $P^{Tb}$ is defined in Lemma \ref{Lem_corresp} and the equation is satisfied in the sense given in Remark \ref{Rem_Modification}.
Let 
\begin{equation*}
A:= \bigcap_{K=1}^{\infty} A^K.
\end{equation*}
Then $A$ is a full measure subset of $\mathbb{R} \backslash \Gamma$ and there exists a full measure subset of $A$, which we will call $\tilde{A}$, such that for every $\gamma \in \tilde{A}$ 
\begin{itemize}
\item $\gamma(\cdot)$ is unique on $[0,T-1/K]$ for every $K=1,2,\dots$ (forwards and backwards)
\item $\gamma(s)$ is a Lebesgue point of $u_x^{Tb}(s,\cdot)$ for almost every $s \in [0,T)$
\item for every $K>0$ there exists $N\ge const(1+K)$ such that \eqref{eq_KN} holds,
\item $u^{Tb}$ satisfies equation \eqref{Eq_PropLTb} along $\gamma(\cdot)$ on $[0,T)$ in the sense of Remark \ref{Rem_Modification}.
\end{itemize}
Since, however, $\gamma \notin \Gamma$ then

\begin{itemize}
\item $u^{Tb}$ does not satisfy 
\begin{equation}
\label{eq_doesnot}
\forall_{\kappa \in (\gamma- \frac 1 N, \gamma) \cup (\gamma,\gamma + \frac 1 N),\kappa(\cdot), s \in [0,T]}  N \ge \omega^{\gamma,\kappa}(s) \ge -N
\end{equation}
for any $N$ (otherwise we would have $\gamma \in \Gamma^N \subset \Gamma$). In other words, for every $N>0$ and every $\epsilon \in (0,T)$ there exists $\kappa \in (\gamma - 1/N, \gamma) \cup (\gamma,\gamma+1/N)$ and a characteristic $\kappa(\cdot)$ of $u^{Tb}$ satisfying $\kappa(0)=\kappa$ and $\tilde{\epsilon} \in (0,\epsilon)$
such that $\omega^{\gamma,\kappa}(T-\tilde{\epsilon}) > N$.

\end{itemize}
{\bf Claim:} 
\begin{equation*}
\{\gamma(T): \gamma \in \tilde{A}\} \cap \bigcup_{R>0} L_R^{unique} = \emptyset.
\end{equation*}
Indeed, otherwise, we would have $\gamma(T) \in L_R^{unique,\tilde{N}}$ for some $R>0, \tilde{N}>0$ and $\gamma \in \tilde{A}$. Fix $\epsilon<R$ such that if $\kappa(\cdot)$ is any characteristic of $u^{Tb}$ then $$\kappa(T-\epsilon) \in (\gamma(T-\epsilon) - 1/(2\tilde{N}) , \gamma(T-\epsilon) + 1/(2\tilde{N}))$$
implies $$\kappa(T) \in (\gamma(T) - 1/\tilde{N}, \gamma(T)+1/\tilde{N}).$$ Existence of such $\epsilon$ follows by finite propagation speed of characteristics, which is bounded by $\sup(|u|)$. Now take $\hat{N}$ so large that $$\kappa \in (\gamma - 1/\hat{N}, \gamma)\cup (\gamma,\gamma+1/\hat{N})$$ implies $$\kappa(T-\epsilon) \in (\gamma(T-\epsilon) - 1/(2\tilde{N}), \gamma(T-\epsilon) + 1/(2\tilde{N})).$$ 
This is possible, since by \eqref{eq_KN} for $K=\lceil 1/\epsilon \rceil$, there exists $\bar{N}$ such that $\omega^{\gamma,\kappa}(s) \le \bar{N}$ for $s \in [0,T-\epsilon]$ and hence $|\kappa(T-\epsilon) - \gamma(T-\epsilon)| \le e^{(T-\epsilon)\bar{N}} |\kappa - \gamma|$. Combining the two implications above we conclude that if $$\kappa \in (\gamma - 1/\hat{N}, \gamma)\cup (\gamma,\gamma+1/\hat{N})$$
then $$\kappa(T) \in (\gamma(T) - 1/\tilde{N}, \gamma(T)) \cup (\gamma(T), \gamma(T) + 1/\tilde{N}).$$
This means that \eqref{eq_doesnot} holds for $N \ge \max(\tilde{N},\bar{N},\hat{N})$, which is a contradiction. Thus, the claim is proven.

The claim implies, by Proposition \ref{Prop_MainCH},
\begin{equation*}
\{\gamma(T): \gamma \in \tilde{A}\} \cap \bigcup_{R>0} S_R = \emptyset,
\end{equation*}
where sets $S_R$ correspond to the original solution $u(t,x)$ as in Proposition \ref{Prop_MainCH}.

Since, however, $\bigcup_{R>0} S_R$ is a full measure subset of $\mathbb{R}$ by Proposition \ref{Prop_MainCH}, the Lebesgue measure of $$B:=\{\gamma(T): \gamma \in \tilde{A}\}$$ is zero.
Using the fact that if $\gamma(\cdot)$ is a characteristic of $u^{Tb}$ then $\gamma(T-\cdot)$ is, by Lemma \ref{Lem_corresp}, a characteristic of $u$ corresponding to the same curve we obtain
\begin{equation}
\label{eq_AeqB}
\tilde{A}(t):=\{\gamma(t): \gamma \in \tilde{A}\} \subset B(T-t),
\end{equation}
where $B(T-t)$ is the thick pushforward along the characteristics of the original solution $u(t,x)$. Note that in \eqref{eq_AeqB} there is no equality, since there might exist other characteristics emanating from $B$. Hence, by Lemma \ref{Lem_limux2} and Lemma \ref{Lem_L1Bt} we obtain 
\begin{itemize}
\item 
\begin{equation*}
\lim_{t \to 0^+} \mathcal{L}^1(B(t)) = 0,
\end{equation*}
\item
\begin{equation}
\label{eq_limL1At}
\lim_{t \to 0^+} \mathcal{L}^1(\tilde{A}(T-t)) = 0,
\end{equation}
\item
\begin{equation*}
\lim_{t \to 0^+} \int_{B(t)} u_x^2(t,x) dx=0,
\end{equation*}
\item
\begin{equation*}
\lim_{t \to 0^+} \int_{\tilde{A}(T-t)} (u^{Tb}_x)^2(T-t,x) dx=0.
\end{equation*}
\end{itemize}

Suppose now $\tilde{A}$ is a set of positive Lebesgue measure (without loss of generality bounded and subset of $[-D,D]$ for some $D>0$). Since $$u_x^{Tb}(t,x) \ge -const(1+1/(T-t)),$$ see \eqref{eq_TbDissip}, we conclude, by Proposition \ref{Prop_ChangeOfV} and Remark \ref{Rem_0to0}, that $\tilde{A}(t)$ has positive Lebesgue measure for every $t \in [0,T)$. 
Let us fix $M$ so large that
\begin{equation}
\label{eq_Mbound}
\begin{cases}\frac 1 2 M^2 > \sup(u^2) + \sup(P),\\
M>2(\sup(u^2)+\sup(P)).
\end{cases}
\end{equation}

{\bf Claim:} There exists $t_0 \in (0,T)$ such that 
$\mathcal{L}^1(\{y \in \tilde{A}(T-t_0): u_x^{Tb}(T-t_0,y)\le -M\} ) > 0.$
 \newline \\
Indeed, otherwise we would have $u_x^{Tb} (s,\gamma(s)) \ge -M$ for every $s \in (0,T)$ and almost every $\gamma \in \tilde{A}$. Hence, recalling Proposition \ref{Prop_ChangeOfV} and the fact that $\tilde{A} \subset L_{(T-1/K),Tb}^{unique}$ (see Definition \ref{def_LTKTb}) for every $K=1,2,\dots$ we obtain
\begin{equation*}
\mathcal{L}^1(\tilde{A}(T-t)) = \int_{\tilde{A}} e^{\int_{0}^{T-t} u_x^{Tb}(s,\gamma(s))ds} d\gamma  
\ge \int_{\tilde{A}} e^{-TM}d\zeta = e^{-TM}\mathcal{L}^1(\tilde{A}).
\end{equation*}
Taking $t \to 0^+$ we obtain, using \eqref{eq_limL1At}, $\mathcal{L}^1(\tilde{A})=0$, which contradicts the assumption that $\mathcal{L}^1(\tilde{A})>0$. Thus the claim is proven.

Let  $\hat{A} = \{\gamma \in \tilde{A}:  {u^{Tb}_x(T-t_0,\gamma(T-t_0))\le -M}\}$, where $t_0$ is obtained from the claim above. 
Then, by Claim, $\mathcal{L}^1(\hat{A}(T-t_0))>0$.

\begin{figure}[h!]
\center
\includegraphics[width=5cm]{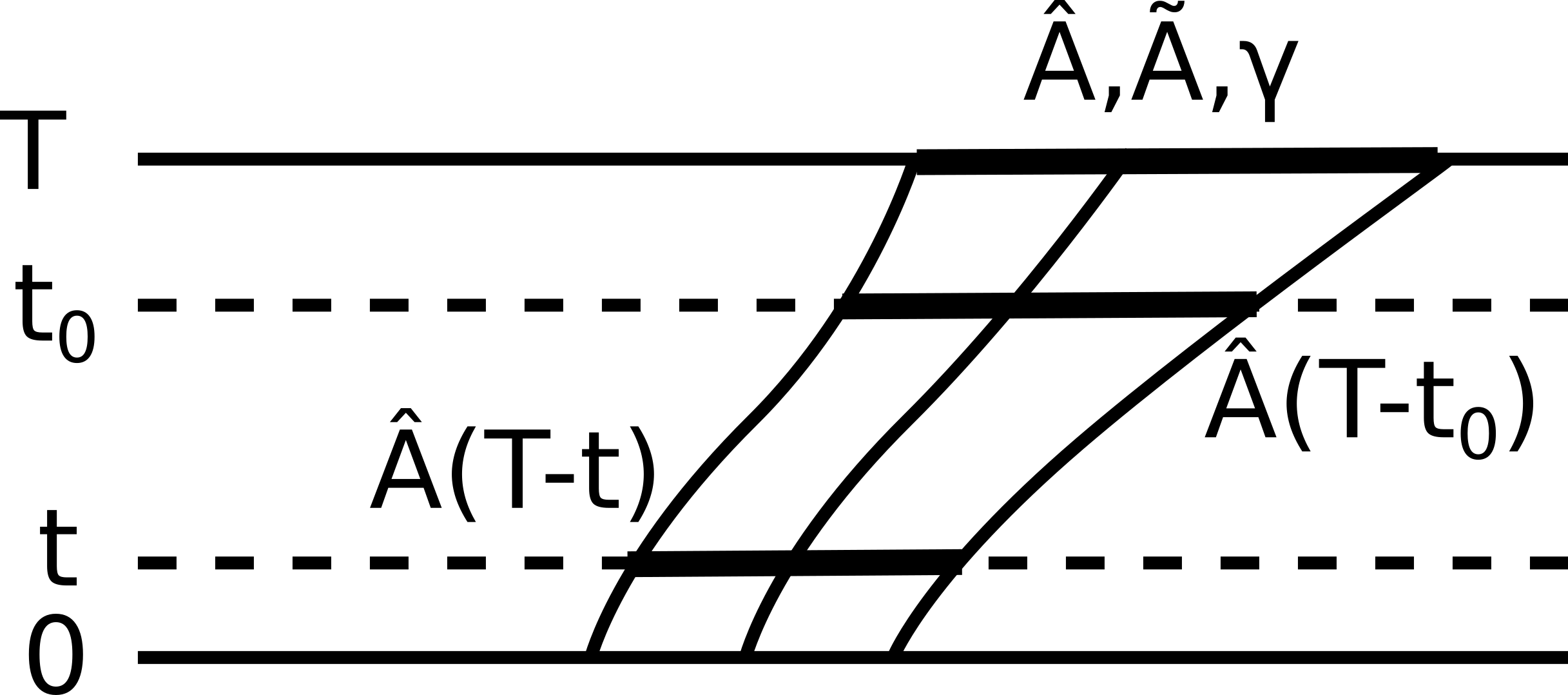}
\caption{Schematic presentation of Part 2 of the proof of Theorem \ref{Th_STT}. The set $\tilde{A}$, being a full measure subset of $\mathbb{R}\backslash \Gamma$ can be visualized at time $T$ (which corresponds to time $0$ for the backward solution $u^{Tb}$). $\hat{A}$ is a subset of $\tilde{A}$ such that for some $t_0$ we have $\mathcal{L}^1(\hat{A}(T-t_0))>0$ and $u_x^{Tb}(T-t_0,x) \le -M$ for every $x \in \hat{A}(T-t_0)$. We obtain a contradiction by proving that on the one hand $\int_{\hat{A}(T-t)}(u_x^{Tb})^2(T-t,y)dy$ converges to $0$ as $t \to 0$, but on the other hand this quantity is bounded away from $0$.}
\label{Fig1}
\end{figure}

Let $\gamma \in \hat{A}$. Computing as in \eqref{eq_ddtv2} we obtain
\begin{equation*}
\frac{d}{ds} v^2M_s' = 2v\dot{v}M_s' + v^2 (vM_s') = vM_s'(2u^2-v^2 -2P+v^2) = 2vM_s'(u^2-P), 
\end{equation*}
where $u=u^{Tb}(s,\gamma(s))$, $v = u^{Tb}_x(s,\gamma(s))$, $P=P^{Tb}(s,\gamma(s))$ and $M_s' = e^{\int_0^s v(r)dr}$. 

Now, equation \eqref{Eq_PropLTb} in combination with \eqref{eq_Mbound} implies $\dot{v} \le 0$  and hence $v \le -M$ for $s \ge s_0:=T-t_0$. Thus, for $s>s_0$
\begin{equation*}
v^2(s) M_s' = v^2(s_0) M_{s_0}' e^{\int_{s_0}^s \frac {2(u^2-P)}{v}ds} 
\end{equation*}
which implies, by \eqref{eq_Mbound},
\begin{equation*}
e^{-T}\le \frac{v^2(s) M_s'}{v^2(s_0)M_{s_0}'}\le e^T.
\end{equation*}
Proposition \ref{Prop_ChangeOfV} yields, for $s > s_0$,

\begin{eqnarray*}
\int_{\hat{A}(s)} (u^{Tb}_x)^2(s,y)dy &=& \int_{\hat{A}} (u^{Tb}_x)^2(s,\gamma(s))M_s'(\gamma)d\gamma \\ &\le& e^T \int_{\hat{A}} (u^{Tb}_x)^2(s_0,\gamma(s_0))M_{s_0}'(\gamma)d\gamma = \int_{\hat{A}(s_0)} (u^{Tb}_x)^2(s_0,y)dy
\end{eqnarray*}
and, similarly,
$$\int_{\hat{A}(s)} (u^{Tb}_x)^2(s,y)dy  \ge e^{-T} \int_{\hat{A}(s_0)} (u^{Tb}_x)^2(s_0,y)dy.$$
Hence, for $t<t_0$,
\begin{equation*}
e^{-T} \le \frac {\int_{\hat{A}(T-t)} (u^{Tb}_x)^2(T-t,y)dy} {\int_{\hat{A}(T-t_0)} (u_x^{Tb})^2(T-t_0,y)dy} \le e^T.
\end{equation*}
Passing to the limit $t \to 0^+$ we obtain that
\begin{eqnarray*}
0=\lim_{t \to 0^+} \int_{\tilde{A}(T-t)} (u^{Tb}_x)^2(T-t,y)dy &\ge& \liminf_{t \to 0^+} \int_{\hat{A}(T-t)} (u^{Tb}_x)^2(T-t,y)dy\\ &\ge& e^{-T} \int_{\hat{A}(T-t_0)} (u^{Tb}_x)^2(T-t_0,y)dy \\ &\ge& e^{-T}M^2 \mathcal{L}^1(\hat{A}(T-t_0)) >0.
\end{eqnarray*}
This gives a contradiction. Hence, $\tilde{A}$ is a null set and, consequently, $\mathcal{L}^1(\mathbb{R} \backslash \Gamma)=0$, which concludes the proof of Part 2.

\begin{corollary}
\label{Cor_globally}
$S_T(T)$ is a full measure set also for $u$ -- a weak solution of \eqref{eq_WeakCH1}-\eqref{eq_WeakCH3} satisfying $u_x \le C$ globally. The proof requires only Part 1 of the Proof of Theorem \ref{Th_STT}.
\end{corollary}

\section {\bf Theory of Bressan-Constantin and proof of Theorem \ref{th_DU}}
\label{Sec_MainProof}

\noindent{\bf Theory of Bressan-Constantin}

In \cite{BC2} A. Bressan and A. Constantin introduced a transformation of the space variable at time $t=0$, $\xi \mapsto \bar{y}(\xi)$ defined implicitly by
\begin{equation}
\label{eq_BCTrans}
\int_0^{\bar{y}(\xi)} (1+\bar{u}_x^2)dx =\xi,
\end{equation}
where $\bar{u}(\cdot)$ is the initial datum of the Camassa-Holm equation.
Using this new variable they proposed the following system of ODEs:
\begin{equation}
\label{eq_BCODE}
\begin{cases}
\frac {\partial \tilde{u}}{\partial t} = -\tilde{P}_x,\\
\frac {\partial \tilde{v}}{\partial t} = \begin{cases} 
(\tilde{u}^2-P)(1+\cos \tilde{v}) - \sin^2 \frac {\tilde{v}} 2 &\mbox{ if } \tilde{v} > -\pi\\
-1 &\mbox{ if } \tilde{v} \le -\pi
\end{cases}\\
\frac{\partial \tilde{q}}{\partial t}=
\begin{cases} 
(\tilde{u}^2+\frac 1 2 - \tilde{P}) \sin \tilde{v} \cdot \tilde{q} &\mbox{ if } \tilde{v} > -\pi\\
0 &\mbox{ if } \tilde{v} \le -\pi
\end{cases}
\end{cases}
\end{equation}
where $\tilde{u}=\tilde{u}(t,\xi), \tilde{v}=\tilde{v}(t,\xi), \tilde{q}=\tilde{q}(t,\xi)$ are unknown functions,
$\tilde{P}, \tilde{P}_x$ are given by
\begin{eqnarray}
\tilde{P}(\xi) &=& \frac  1 2 \int_{\{\tilde{v}(\xi')>-\pi\}} \exp\left\{-\left|\int_{\{s \in [\xi,\xi'], \tilde{v}(s)>-\pi\}} \cos^2\frac {\tilde{v}(s)}{2} \tilde{q}(s)\right|\right\} \nonumber\\
&&\left[\tilde{u}^2(\xi')\cos^2 \frac{\tilde{v}(\xi')}{2} + \frac 1 2 \sin^2 \frac {\tilde{v}(\xi')}{2} \right] \tilde{q}(\xi') d\xi', \label{eq_Pxidziura}\\
\tilde{P}_x(\xi) &=& \frac {1}{2} \left(\int_{\{\xi'>\xi, \tilde{v}(\xi')>-\pi\}} - \int_{\{\xi'<\xi, \tilde{v}(\xi')>-\pi\}}\right) \exp\left\{-\left|\int_{\{s \in [\xi,\xi'], \tilde{v}(s)>-\pi\}} \cos^2\frac {\tilde{v}(s)}{2} \tilde{q}(s)\right|\right\}\nonumber \\
&&\left[\tilde{u}^2(\xi')\cos^2 \frac{\tilde{v}(\xi')}{2} + \frac 1 2 \sin^2 \frac {\tilde{v}(\xi')}{2} \right] \tilde{q}(\xi') d\xi'\label{eq_Pxxidziura}
\end{eqnarray}
and the initial condition is of the form
\begin{equation*}
\begin{cases}
\tilde{u}(0,\xi) = \bar{u}(\bar{y}(\xi)),\\
\tilde{v}(0,\xi) = 2 \arctan(\bar{u}_x(\bar{y}(\xi))),\\
\tilde{q}(0,\xi) = 1.
\end{cases}
\end{equation*}

The solutions $(\tilde{u},\tilde{v},\tilde{q})$ of \eqref{eq_BCODE} allowed to construct dissipative solutions $u(t,x)$ of the Camassa-Holm equation using identities 
\begin{eqnarray*}
\tilde{u} = \tilde{u}(t,\xi)&=& u(t,y(t,\xi)),\\
\tilde{u}_x = \tilde{u}_x(t,\xi)&=& u_x(t,y(t,\xi)),\\
\tilde{v} &=& 2 \arctan \tilde{u}_x, \\
\tilde{q} &=& (1+\tilde{u}_x^2)\frac {\partial y}{\partial \xi} \\
\frac {\partial}{\partial t} y(t,\xi) &=& \tilde{u}(t,\xi),\\
y(0,\xi) &=& \bar{y}(\xi),
\end{eqnarray*}
which were used to derive \eqref{eq_BCODE}, see \cite{BC2} for details. The solutions thus obtained form a semigroup. It has not been known, however, whether this procedure is the only one which yields dissipative solutions and thus the question of uniqueness has remained open. In this section we will go in the opposite direction to \cite{BC2}, i.e. we will show that every dissipative solution $u(t,x)$ of \eqref{eq_WeakCH1}-\eqref{eq_WeakCH3} admits a unique representation $(\tilde{u},\tilde{v},\tilde{q})$, which solves \eqref{eq_BCODE}. This, due to local well-posedness of \eqref{eq_BCODE} proven in \cite{BC2}, will result in uniqueness of dissipative solutions.
\begin{remark}
Note that $\tilde{v}$ above has different meaning from the meaning of $v$ in the preceding sections; the reason for this coincidence is that we wanted to keep the notations both from \cite{BC2} and \cite{CHGJ}. As in this section we will not, however, use the meaning of $v$ from preceding sections of the present paper, this should not lead to confusion. 
\end{remark}

\begin{remark}
The local well-posedness of \eqref{eq_BCODE} was obtained in \cite{BC2} by rewriting the system \eqref{eq_BCODE} as 

\begin{eqnarray}
\frac {\partial}{\partial t} U(t,\xi) &=& F(U(t,\xi)) + G(\xi,U(t,\cdot)), \label{eq_UU1}\\
U(0,\xi) &=& \bar{U}(\xi),\label{eq_UU2} 
\end{eqnarray}
where $U=(\tilde{u},\tilde{v},\tilde{q}) \in \mathbb{R}^3$ and $F,G$ are defined accordingly (see formulas (4.1)-(4.4) in \cite{BC2}) and considering it as an ordinary differential equation on the space $L^{\infty}(\mathbb{R};\mathbb{R}^3)$. Local existence and uniqueness follows by considering a contraction $\mathcal{P}: \mathcal{D}\to \mathcal{D}$ on the domain $\mathcal{D} \subset C([0,T],L^{\infty}(\mathbb{R},\mathbb{R}^3))$ given by
\begin{equation*}
(\mathcal{P}(U))(t,\xi):= \bar{U} + \int_0^t [F(U(\tau,\xi))+G(\xi,U(\tau,\cdot))]d\tau.
\end{equation*}
The domain $\mathcal{D}$ consists of all continuous mappings $t \mapsto U(t)=(\tilde{u}(t),\tilde{v}(t),\tilde{q}(t))$ from $[0,T]$ into $L^{\infty}(\mathbb{R},\mathbb{R}^3)$ satisfying:
\begin{eqnarray*}
U(0)&=&\bar{U},\\
\|U(t)-U(s)\|_{L^{\infty}} &\le& 2\kappa^* |t-s|, \\
\tilde{v}(t,\xi)-\tilde{v}(s,\xi) &\le& -\frac {t-s}{2} \qquad \xi \in \Omega^{\delta}, 0\le s <t \le T,
\end{eqnarray*}
where $\kappa^*$ is a constant dependent on the apriori estimates of the system \eqref{eq_UU1}-\eqref{eq_UU2}, $\Omega^{\delta} := \{\xi \in \mathbb{R}; \bar{v}(\xi) \in (-\pi,\delta-\pi]\}$ and $\delta$ is so small that the Lebesgue measure of $\Omega^\delta$ is small compared to constants related to $F,G$ and $v(\xi) \in (-\pi,\delta-\pi]$ implies $\frac {\partial} {\partial t} v(t,\xi) \le -1/2$. The main difficulty, overcome in \cite{BC2}, lies in the discontinuity of the right-hand side, see also \cite{BSh} for a general theory.
\end{remark}

\noindent{\bf Proof of Theorem \ref{th_DU}.}
The proof is divided into 3 parts:
\begin{itemize}
\item In Part 1, based on an arbitrary dissipative weak solution of \eqref{eq_WeakCH1}-\eqref{eq_WeakCH3} we define a triple $(\tilde{u},\tilde{v},\tilde{q})$, which is supposed to satisfy \eqref{eq_BCODE}.
\item In Part 2 we show that the triple $(\tilde{u},\tilde{v},\tilde{q})$ defined in Part 1 indeed satisfies \eqref{eq_BCODE}.
\item In Part 3, using the well-posedness result from \cite{BC2}, we show that this representation is unique and dependent only on the initial data. This, in combination with Theorem \ref{Th_STT}, leads directly to uniqueness of dissipative solutions.
\end{itemize}

{\bf Part 1.} Let us fix $u(t,x)$ -- arbitrary dissipative weak solution of \eqref{eq_WeakCH1}-\eqref{eq_WeakCH3} and $T<T_{max}$ -- small enough, see \eqref{eq_defTmax}. Then, 
by Proposition \ref{Prop_MainCH}, the equation 
\begin{equation}
\label{eq_uxdot}
\dot{u_x}(t) = u^2(t) - \frac 1 2 u_x^2(t) - P
\end{equation}
is satisfied on $[0,T]$ for every $\zeta \in S_T$, where $u_x(t)=u_x(t,\zeta(t))$, $u(t)=u(t,\zeta(t))$ and
$$P=P(t,\zeta(t))= \frac 1 2 \int_{-\infty}^{\infty} e^{-|\zeta(t)-x|} \left(u^2(t,x) + \frac 1 2 u_x^2(t,x)\right)dx.$$
Clearly, the equation
\begin{equation*}
\dot{u}(t) = -P_x,
\end{equation*}
where 
$$P_x=P_x(t,\zeta(t))= \frac 1 2 \left(\int_{\zeta(t)}^{\infty} - \int_{-\infty}^{\zeta(t)}\right) e^{-|\zeta(t)-x|} \left(u^2(t,x) + \frac 1 2 u_x^2(t,x)\right)dx$$
is also satisfied on $[0,T]$ as it is satisfied along every characteristic by Proposition \ref{Prop_Daf}. 

Before we proceed, let us note that dissipative solutions benefit from unique characteristics.
\begin{lemma}
\label{Lem_charunique}
Let $u$ be a dissipative weak solution of \eqref{eq_WeakCH1}-\eqref{eq_WeakCH3}. Then for every $\zeta \in \mathbb{R}$ there exists a unique (forwards) characteristic emanating from $\zeta$.
\end{lemma}
\begin{proof}
Let $\zeta^l(\cdot)$ and $\zeta^r(\cdot)$ be the leftmost and rightmost, respectively, characteristics emanating from $\zeta$. Suppose $\zeta^l(t) < \zeta^r(t)$. Then $\zeta \notin S_t$. By Theorem \ref{Th_STT} we have $\mathcal{L}^1([\zeta^l(t),\zeta^r(t)])=0$, which gives contradiction. Hence, $\zeta^l(t)=\zeta^r(t)$ for $t$ small enough. Since this reasoning works for arbitrary initial time $t_0$, we conclude that $\zeta^l(\cdot)=\zeta^r(\cdot)$ globally.
\end{proof}
Denote 
\begin{equation*}
y(t,\xi) := \bar{y}(\xi)(t),
\end{equation*}
where $\bar{y}(\xi)(\cdot)$ is the unique forwards characteristic of $u(t,x)$ emanating from $\bar{y}(\xi)$, i.e. $\bar{y}(\xi)(\cdot)$ satisfies $\bar{y}(\xi)(0) = \bar{y}(\xi)$ and $\frac {d}{dt} \bar{y}(\xi)(t) = u(t,\bar{y}(\xi)(t))$. Next we define $(\tilde{u}(t,\xi),\tilde{v}(t,\xi),\tilde{q}(t,\xi))$ for $(t,\xi) \in [0,T]\times \mathbb{R}$ in terms of our dissipative solution $u(t,x)$. The definition depends on whether $(t,\xi)$ satisfies $\bar{y}(\xi) \in S_t$.

Suppose first that $(t,\xi)$ is such that $\bar{y}(\xi) \in S_t$. Then we set
\begin{eqnarray}
\tilde{u}(t,\xi)&:=&  u(t,\bar{y}(\xi)(t)),\label{eq_uuu}\\
\tilde{v}(t,\xi)&:=& 2\arctan(u_x(t,\bar{y}(\xi)(t))),\label{eq_vvv}\\
\tilde{q}(t,\xi)&:=&  e^{\int_0^t u_x(s,\bar{y}(\xi)(s))ds} \frac {1+u_x^2(t,\bar{y}(\xi)(t))}{1+{u}_x^2(0,\bar{y}(\xi))}\label{eq_qqq}.
\end{eqnarray}
Note that we have $\tilde{u}(0,\xi) = u(0,\bar{y}(\xi)), \tilde{v}(0,\xi)=2\arctan(u_x(0,\bar{y}(\xi))), \tilde{q}(0,\xi)=1.$
\begin{remark}
\label{Rem_cont}
The functions $t \mapsto \tilde{v}(t,\xi)$, $t \mapsto \tilde{q}(t,\xi)$, defined in \eqref{eq_uuu}-\eqref{eq_qqq} may not be continuous due to the fact that $u_x(t,\bar{y}(\xi)(t))$ might not be a limit of difference quotients for every $t$. However, $u_x,\tilde{v},\tilde{q}$ can be, for every $\xi$, modified on a set of $t$ of measure $0$, so that they become continuous and it is this modification that we will be using. Note also that for every fixed $t$ the set of $\bar{y}(\xi)(t)$ such that this modification is necessary has Lebesgue measure $0$, compare Remark \ref{Rem_Modification}.
\end{remark}

\begin{remark} Due to Theorem \ref{Th_STT} we have that $S_t(t)$ is a set of full Lebesgue measure for every $t>0$. Thus, although for fixed $t$ the set $\{\xi: \bar{y}(\xi) \in S_t\}$ may not be of full measure (this happens, for instance, when characteristics collide before time $t$), certainly the set $\{\bar{y}(\xi)(t): \bar{y}(\xi)\in S_t\}$ has full Lebesgue measure. 
\end{remark}

Next, consider $(t,\xi) \in [0,T]\times \mathbb{R}$ such that $\bar{y}(\xi) \notin S_t.$ For any $\zeta \in \mathbb{R}$ define the breaking time $T_{br}(\zeta)$ by 
$$T_{br}(\zeta):= \sup\{T>0: \zeta\in S_T\}.$$ 
Similarly, define $T_{br}(\xi)$ by
$$T_{br}(\xi):= T_{br}(\bar{y}(\xi)) = \sup\{T>0: \bar{y}(\xi)\in S_T\}.$$ 
\begin{proposition}
\label{Prop_Tbr}
$T_{br}$ has the following properties
\begin{enumerate}[i)]
\item $T_{br}(\zeta)>0$ for almost every $\zeta \in \mathbb{R}$ and $T_{br}(\xi)>0$ for almost every $\xi \in \mathbb{R},$
\item If $T_{br}(\xi)<\infty$ then   $\lim_{t \uparrow T_{br}(\xi)} (\tilde{u},\tilde{v},\tilde{q})(t,\xi)$ exists. 
\item For fixed $T$ we have $\mathcal{L}^1(\{\zeta: 0<T_{br}(\zeta)\le T, \lim_{t \uparrow T_{br}(\zeta)} u_x(t,\zeta(t))\neq -\infty\})=0$ and, similarly, $\mathcal{L}^1(\{\xi: T_{br}(\xi)\le T, \lim_{t \uparrow T_{br}(\xi)} u_x(t,\xi)\neq -\infty\})=0.$
\end{enumerate}
\end{proposition}
\begin{proof}
Property i) follows immediately from Proposition \ref{Prop_MainCH} and the change of variables \eqref{eq_BCTrans}. To prove ii) observe that:
\begin{itemize}
\item $u$ is continuous,
\item $u_x$ satisfies the differential  equation \eqref{eq_uxdot} on $[0,T_{br}(\xi))$ and hence,  compare Remark \ref{Rem_cont}, there exists the limit $\lim_{t \uparrow T_{br}(\xi)} u_x(t,\bar{y}(\xi)(t))$ (which may be equal $-\infty$). Thus the limit $\lim_{t \uparrow T_{br}(\xi)} \tilde{v}(t,\xi)$ exists (possibly equal $-\pi$),
\item $\tilde{q}$ satisfies the differential equation
\begin{equation*}
\partial_t \tilde{q} (t,\xi) = \frac {\tilde{u}_x(t,\xi)}{1+\tilde{u}^2_x(t,\xi)} (1+2\tilde{u}^2(t,\xi) - 2P(t,\bar{y}(\xi)(t))) \tilde{q}(t,\xi),
\end{equation*}
see derivation \eqref{eq_q} below; hence $\frac \partial {\partial t}{\log(\tilde{q})}$ is bounded and the limit $\lim_{t \uparrow T_{br}(\xi)} \tilde{q}(t,\xi)$ exists; moreover, since $\tilde{q}(0,\xi)=1$, there exists a global constant $\tilde{C}$ such that 
\begin{equation}
\label{eq_C1C}
1/\tilde{C}<\tilde{q}(t,\xi)<\tilde{C}
\end{equation}
for all $t \in [0,T_{br}(\xi))$.
\end{itemize}
To prove iii) let us consider the following sets:
\begin{equation*}
A^M:= \{\zeta: T_{br}(\zeta) \le T, \forall_{t \in [0,T_{br}(\zeta))} u_x(t,\zeta(t))> -M\}.
\end{equation*}
Observe that by ii)
\begin{equation*}
\bigcup_{M=1}^{\infty} A^M = \{\zeta: T_{br}(\zeta)\le T, \lim_{t \uparrow T_{br}(\zeta)} u_x(t,\zeta(t))\neq -\infty\}.
\end{equation*}
We will show that $\mathcal{L}^1(A^M)=0$ for every fixed $M>0$. For this purpose, let us first recall some definitions and results from \cite{CHGJ}.
\begin{definition}[Definition 5.5 from \cite{CHGJ}]
\label{def_6755}
For $0\le t < T_{max}$, where $T_{max}$ is given by \eqref{eq_defTmax}, we define
\begin{eqnarray*}
\Omega(t)&:=& \sqrt{LC} \tan(\sqrt{LC}t - \pi/2), \\
I_t &:=& \left\{\zeta \in \mathbb{R}: u_x(0,\zeta)=\lim_{h\to 0} \frac {u(0,\zeta+h) - u(0,\zeta)}{h}, u_x(0,\zeta)>\Omega(t) \right\},\\
I_t^{unique,N} &:=& \{\zeta \in I_t: \forall_{\eta \in (\zeta-1/N,\zeta)\cup (\zeta,\zeta+1/N), s \in [0,t)} -N \le \omega(s)\le N\}, \\
I_t^{unique} &:=& \bigcup_{N=1}^\infty I_t^{unique,N},
\end{eqnarray*}
where $C$ is defined by \eqref{eq_defC}, $L$ can be chosen $1$ for the Camassa-Holm equation, compare Proof of Lemma \ref{LemA45}, and $\omega(s) = \frac {u(s,\eta(s))-u(s,\zeta(s)}{\eta(s)-\zeta(s)}$ with $\eta(\cdot)$ any characteristic of $u$ emanating from $\eta$. 
\end{definition}
\begin{lemma}[Lemma 5.7 from \cite{CHGJ}]
\label{Lem_57chgj}
For $t<T_{max}$ the set $I_t \backslash I_t^{unique}$ has Lebesgue measure $0$.
\end{lemma}
\begin{lemma}[Lemma 5.9 from \cite{CHGJ}]\label{Lem_59chgj}
 For 
$t<T_{max}$ the set $I_t^{unique} \backslash L_t^{unique}$ has Lebesgue measure $0$, where $L_t^{unique}$ is given by Definition \ref{Def_Lt}.
\end{lemma}

Now, Definition \ref{def_6755} implies $A^M \subset I_{\Delta^M}$, where $\Delta^M := \Omega^{-1}(-M)$. Hence by Lemma \ref{Lem_57chgj} we obtain that almost every $\zeta \in A^M$ belongs in fact to $I_{\Delta^M}^{unique}$ and hence almost every $\zeta \in A^M$ belongs, by Lemma \ref{Lem_59chgj} and Proposition \ref{Prop_MainCH}, to $S_{\Delta^M}$. This 
implies $T_{br}(\zeta)\ge \Delta^M$ for  every $\zeta \in A_M \cap S_{\Delta_M}$. Consequently, 
$T_{br}(\zeta) \ge \Delta_M$ for almost every $\zeta \in A_M$.

 Consider now initial time $t_0=\frac 3 4\Delta^M$ and shifted in time weak solution of the Camassa-Holm equation, $$u^{t_0}(t,x):=u(t+t_0,x).$$ Repeating the reasoning above, we show  that $T^{t_0}_{br}(\tilde{\zeta})\ge \Delta^M$ for a.e. $\tilde{\zeta} \in (A_M \cap S_{\Delta^M})(\frac 3 4 \Delta^M)$, where $(A_M \cap S_{\Delta^M})(\frac 3 4 \Delta^M)$ is the thick pushforward, see Definition \ref{Def_Thick}, and $T_{br}^{t_0}$ is $T_{br}$ related to $u^{t_0}$.  Denote $$B:=(A_M\cap S_{\Delta^M})(3/4 \Delta^M) \cap\{\tilde{\zeta} : T_{br}^{t_0}(\tilde{\zeta})<\Delta^M\}.$$ Then, using Proposition \ref{Prop_ChangeOfV},
\begin{eqnarray*}
0=\mathcal{L}^1(B) &=& \int_{(A_M \cap S_{\Delta^M})\cap \{\zeta: T_{br}^{t_0}(\zeta(\frac 3 4\Delta^M)) < \Delta^M\}} e^{\int_0^{\frac 3 4 \Delta^M} u_x(s,\zeta(s))ds}d\zeta\\ &\ge& e^{-\frac 3 4 \Delta^M M} \mathcal{L}^1((A_M \cap S_{\Delta^M})\cap \{\zeta: T_{br}^{t_0}(\zeta(3/4\Delta^M)) < \Delta^M\}).
\end{eqnarray*}
Hence,
$$\mathcal{L}^1((A_M \cap S_{\Delta^M})\cap \{\zeta: T_{br}^{t_0}(\zeta(3/4\Delta^M)) < \Delta^M\}) = 0.$$
Consequently, using Lemma \ref{Lem_610} below, $T_{br}(\zeta)\ge 1 \frac 3 4\Delta^M$ for a.e. $\zeta \in A_M$. 

In the next step, we show simiarly that $T_{br}(\zeta)\ge 2 \frac 1 2\Delta^M$ for a.e. $\zeta \in A_M$  and iterating this construction, we conclude.
\end{proof}
\begin{lemma}
\label{Lem_610}
Let $u$ be a weak solution of \eqref{eq_WeakCH1}-\eqref{eq_WeakCH3}. If $\zeta \in S_{t_0}$ and $\zeta(T) \in S_T^{t_0}$, where $S_T^{t_0}$ is the set $S_T$ corresponding to the shifted in time solution $u^{t_0}(t,x):=u(t+t_0,x)$, then $\zeta \in S_{t_0+T}$ (except possibly for a set of $\zeta$ of measure $0$).
\end{lemma}
\begin{proof}
If $\zeta \in S_{t_0}$ then $\zeta \in L_{t_0}^{unique,N}$ for some $N>0$. Similarly, $\zeta(t_0) \in L_T^{t_0,unique,\tilde{N}}$ for some $\tilde{N}$, where $L_T^{t_0,unique,\tilde{N}}$ is the set $L_T^{unique,\tilde{N}}$ corresponding to $u^{t_0}$. Let $\hat{N} \ge \max\{N,\tilde{N}\}$ be so large that $$((\zeta-1/\hat{N})^l(t_0),(\zeta+1/\hat{N})^r(t_0)) \subset (\zeta(t_0)-1/\tilde{N}, \zeta(t_0)+1/\tilde{N}).$$ Then $\zeta \in L_{t_0+T}^{unique,\hat{N}}$, by Definition \ref{Def_Lt}. Consequently, $\zeta \in L_{t_0+T}^{unique}$ and hence, by Proposition \ref{Prop_MainCH}, $\zeta \in S_{t_0+T}$ (up to a set of $\zeta$ of measure $0$).
\end{proof}

\begin{remark}
Proposition \ref{Prop_Tbr}iii) shows that the breaking of waves in dissipative solutions of the Camassa-Holm equation is only due (up to a set of measure $0$) to blow-up of $u_x$.
\end{remark}

Now, fix $\xi$. We define the triple $(\tilde{u},\tilde{v},\tilde{q})(t,\xi)$ for $t\ge T_{br}(\xi)$ (recall that for $t<T_{br}(\xi)$ it is already defined by \eqref{eq_uuu}-\eqref{eq_qqq}) as the solution of equations \eqref{eq_BCODE} with initial condition
\begin{equation}
\label{eq_IC1}
(\tilde{u},\tilde{v},\tilde{q})(T_{br}(\xi),\xi):=  \lim_{t \uparrow T_{br}(\xi)} (\tilde{u},\tilde{v},\tilde{q})(t,\xi), 
\end{equation}
if  $T_{br}(\xi)>0$ and
\begin{equation}
\label{eq_IC2}
(\tilde{u},\tilde{v},\tilde{q})(0,\xi):=(u(0,\bar{y}(\xi)),2\arctan(u_x(0,\bar{y}(\xi))),1)
\end{equation}
otherwise.
The limit in \eqref{eq_IC1} is well-defined by Proposition \ref{Prop_Tbr}. To solve the equations \eqref{eq_BCODE} we: 
\begin{enumerate}
\item Define for every $(t,\xi)$ 
\begin{eqnarray}
\tilde{P}(t,\xi):=P(t,y(t,\xi))&=&\frac 1 2 \int_{-\infty}^{\infty} e^{-|y(t,\xi)-x|} \left(u^2(t,x)+\frac 1 2 u_x^2(t,x)\right)dx,\label{eq_defPtxi}\\
\tilde{P}_x(t,\xi):=P_x(t,y(t,\xi))&=& \frac 1 2 \left(\int_{y(t,\xi)}^{\infty} - \int_{-\infty}^{y(t,\xi)}\right) \nonumber\\ &&\qquad\qquad e^{-|y(t,\xi)-x|} \left(u^2(t,x) + \frac 1 2 u_x^2(t,x)\right)dx.
\label{eq_defPxtxi}
\end{eqnarray}
\item Integrate in time, for every fixed $\xi$, the first equation of \eqref{eq_BCODE} for $\frac {\partial \tilde{u}}{\partial t}$ with initial condition \eqref{eq_IC1}-\eqref{eq_IC2}. This is possible, since $\tilde{P}_x$ has been defined in \eqref{eq_defPxtxi}.

\item For every fixed $\xi$ solve the second equation of \eqref{eq_BCODE} for $\frac {\partial \tilde{v}}{\partial t}$ with initial condition \eqref{eq_IC1}-\eqref{eq_IC2}. This is possible, since $\tilde{u}$ has been defined in the previous step, $\tilde{P}$ is defined by \eqref{eq_defPtxi}, and the right-hand side is Lipschitz in the variable $\tilde{v}$.
\item For every fixed $\xi$ solve the third equation of \eqref{eq_BCODE} for $\frac {\partial \tilde{q}}{\partial t}$  with initial condition \eqref{eq_IC1}-\eqref{eq_IC2}. This is possible, since $\tilde{P},\tilde{P}_x$ are well defined and $\tilde{u},\tilde{v}$ have already been defined in the previous steps.
\end{enumerate}

Thus, we have defined $(\tilde{u},\tilde{v},\tilde{q})$ for all $(t,\xi) \in [0,T] \times \mathbb{R}$, based on the arbitrary dissipative weak solution $u(t,x)$ of \eqref{eq_WeakCH1}-\eqref{eq_WeakCH3}. 

{\bf Part 2.} Now, we show that the triple $(\tilde{u},\tilde{v},\tilde{q})(t,\xi)$ defined in Part 1 satisfies the system of equations \eqref{eq_BCODE}. We begin by showing that $\tilde{P}(t,\xi),\tilde{P}_x(t,\xi)$ defined in \eqref{eq_defPtxi}-\eqref{eq_defPxtxi} coincide with formulas for $\tilde{P},\tilde{P}_x$ given in \eqref{eq_Pxidziura}-\eqref{eq_Pxxidziura}.

\begin{lemma}
\label{Lem_Pequivtilde}
$\tilde{P}(t,\xi)$ and $\tilde{P}_x(t,\xi)$ defined in \eqref{eq_defPtxi} and \eqref{eq_defPxtxi}, respectively, satisfy formulas \eqref{eq_Pxidziura},\eqref{eq_Pxxidziura}, respectively.
\end{lemma}
\begin{proof}
 We show the full proof for $\tilde{P}$ only, the proof for $\tilde{P}_x$ being similar. $\tilde{P}$, by definition \eqref{eq_defPtxi}, is given by
\begin{equation*}
\tilde{P}(t,\xi):=\frac 1 2 \int_{-\infty}^{\infty} e^{-|y(t,\xi)-x|} \left(u^2(t,x)+\frac 1 2 u_x^2(t,x)\right)dx.
\end{equation*}
By Theorem \ref{Th_STT} the set $S_t(t)$ is of full Lebesgue measure and hence the above formula is equivalent to
\begin{equation*}
\tilde{P}(t,\xi)=\frac 1 2 \int_{S_t(t)} e^{-|y(t,\xi)-x|} \left(u^2(t,x)+\frac 1 2 u_x^2(t,x)\right)dx.
\end{equation*}
In the remaining part of the proof we use the fact that when we restrict ourselves to the set $S_t$, the algebraic manipulations can be performed as if the solution $u(t,x)$ and all the functions involved were Lipschitz continuous, similarly as the manipulations performed at the beginning of \cite[Section 3]{BC2} to derive equations \eqref{eq_BCODE}.

Let us change the variable $x=y(t,\xi')$. Since 
\begin{equation*}
\bold{1}_{x \in S_t(t)} = \bold{1}_{y(t,\xi') \in S_t(t)} = \bold{1}_{\bar{y}(\xi')(t) \in S_t(t)} = \bold{1}_{\bar{y}(\xi')\in S_t}
\end{equation*}
we obtain that for $x \in S_t(t)$ there exists a unique $\xi'$ such that $x=y(t,\xi')$ and
\begin{eqnarray*}
\tilde{P}(t,\xi) &=& \frac 1 2 \int_{\{\bar{y}(\xi') \in S_t\}} e^{-|y(t,\xi) - y(t,\xi')|} \left(\tilde{u}^2(t,\xi') + \frac 1 2 \tilde{u}_x^2(t,\xi')\right) \frac {\partial x}{\partial{\bar{y}}} \frac{\partial \bar{y}}{\partial \xi'} d\xi'\\
&=& \frac 1 2 \int_{\{\bar{y}(\xi') \in S_t\}} e^{-|y(t,\xi) - y(t,\xi')|}\\&&\qquad\qquad \left(\tilde{u}^2(t,\xi') + \frac 1 2 \tilde{u}_x^2(t,\xi')\right) \tilde{q}(t,\xi') \frac{1+\tilde{u}_x^2(0,\xi')}{1+\tilde{u}_x^2(t,\xi')}\frac{1}{1+\tilde{u}_x^2(0,\xi')} d\xi'\\
&=& \frac 1 2 \int_{\{\bar{y}(\xi') \in S_t\}} e^{-|y(t,\xi) - y(t,\xi')|} \\&&\qquad\qquad\left(\tilde{u}^2(t,\xi')\cos^2\left(\frac {\tilde{v}(t,\xi')}{2}\right) + \frac 1 2 \sin^2\left(\frac{\tilde{v}(t,\xi')}{2}\right)\right) \tilde{q}(t,\xi')d\xi',
\end{eqnarray*}
where $\tilde{u}_x^2(t,\xi):= u_x^2(t,\bar{y}(\xi)(t))$ and in the last equality we used the algebraic identities 
\begin{eqnarray*}
\frac{1}{1+\tilde{u}_x^2(t,\xi')} = \cos^2\left(\frac {\tilde{v}(t,\xi')}{2}\right),\\
\frac{\tilde{u}_x^2(t,\xi')}{1+\tilde{u}_x^2(t,\xi')} = \sin^2\left(\frac {\tilde{v}(t,\xi')}{2}\right),
\end{eqnarray*}
which follow from the definition of $\tilde{v}$, \eqref{eq_vvv}, compare  identities (3.5)-(3.6) in \cite{BC2}. Above, we used also \eqref{eq_BCTrans}, by which $\frac {\partial \bar{y}}{\partial \xi'} = \frac {1}{1+\tilde{u}_x^2(0,\xi)}$ (in particular the change of variables $\xi' \to \bar{y}$ is Lipschitz)
and the fact that the change of variables $\bar{y} \to x$ is well-defined by Proposition \ref{Prop_ChangeOfV} and for $\bar{y}(\xi') \in S_t$ we have 
$$\partial x/\partial{\bar{y}}=M_t' = e^{\int_0^t u_x(s,\bar{y}(\xi')(s))ds}= \tilde{q}(t,\xi') \frac{1+\tilde{u}_x^2(0,\xi')}{1+\tilde{u}_x^2(t,\xi')},$$ see \eqref{eq_qqq}. To finish the proof we need to show that
\begin{enumerate}[i)]
\item The integral $\int_{\{\bar{y}(\xi') \in S_t\}} \dots$ can be equivalently  performed over the set $\{\tilde{v}(\xi')>-\pi\}=\{\xi': \tilde{v}(t,\xi')>-\pi\}$,
\item $|y(t,\xi)-y(t,\xi')| = \int_{\{s \in [\xi,\xi'], \tilde{v}(s)>-\pi\}} \cos^2\frac {\tilde{v}(s)}{2} \tilde{q}(s)ds.$
\end{enumerate}
To prove i) we observe that, by \eqref{eq_vvv},
\begin{equation*}
\{\bar{y}(\xi') \in S_t\} \subset \{\tilde{v}(t,\xi')>-\pi\}.
\end{equation*} 
Hence, we can write
\begin{equation*}
\{\tilde{v}(t,\xi')>-\pi\} = A_1 \cup A_2,
\end{equation*}
where
\begin{eqnarray*}
A_1 &=& \{\bar{y}(\xi') \in S_t\},\\
A_2 &=& \{\bar{y}(\xi') \notin S_t\} \cap \{\tilde{v}(t,\xi')>-\pi\},\\
\end{eqnarray*}
and $A_1 \cap A_2 = \emptyset$. Note that if $\bar{y}(\xi') \notin S_t$ then $\{T_{br}(\xi')\le t\}$. Moreover, if $\tilde{v}(t,\xi')>-\pi$ then also $\tilde{v}(T_{br}(\xi'),\xi')>-\pi$ by second equation of \eqref{eq_BCODE}. More precisely, if $\tilde{v}(T_{br}(\xi'),\xi')=-\pi$ then for every $s \in [T_{br}(\xi'),t]$ we have, by definition of $\tilde{v}$ in Part 1, that $\tilde{v}(s)\le -\pi$. Proposition \ref{Prop_Tbr}iii) yields now $\mathcal{L}^1(A_2)=0$, which concludes the proof of i).
To prove ii), we assume without loss of generality that $y(t,\xi)>y(t,\xi')$ and calculate
\begin{eqnarray*}
|y(t,\xi) - y(t,\xi')| &=& \int_{[y(t,\xi'),y(t,\xi)]} dx = \int_{[y(t,\xi'),y(t,\xi)] \cap S_t(t)} dx \\&=&  \int_{\{s: s \in [\xi',\xi], \bar{y}(s) \in S_t\}} \frac {\partial x}{\partial{\bar{y}}} \frac {\partial \bar{y}}{\partial{s}}ds = \int_{\{s: s \in [\xi',\xi], \bar{y}(s) \in S_t\}} \cos^2\frac{\tilde{v}(s)}{2} \tilde{q}(s) ds, 
\end{eqnarray*}
where $\bar{y}=\bar{y}(s)$ and, as before, the change of variables on $S_t$ is allowed. We conclude, using i).
The proof for $\tilde{P}_x$ is similar since 
\begin{eqnarray*}
\tilde{P}_x(t,\xi)&:=&\frac 1 2 \int_{-\infty}^{\infty} sgn(x-y(t,\xi)) e^{-|y(t,\xi)-x|} \left(u^2(t,x)+\frac 1 2 u_x^2(t,x)\right)dx\\
&=&\frac 1 2 \int_{\{\bar{y}(\xi') \in S_t\}} sgn(y(t,\xi')-y(t,\xi))e^{-|y(t,\xi) - y(t,\xi')|}\\ && \qquad\qquad\qquad\qquad\left(\tilde{u}^2(t,\xi') + \frac 1 2 \tilde{u}_x^2(t,\xi')\right) \frac {\partial x}{\partial{\bar{y}}} \frac{\partial \bar{y}}{\partial \xi'} d\xi'\\
&=&\frac 1 2 \left(\int_{\{\bar{y}(\xi') \in S_t\} \cap \{\xi'>\xi\}} - \int_{\{\bar{y}(\xi') \in S_t\} \cap \{\xi'<\xi\}} \right) \\&& \qquad e^{-|y(t,\xi) - y(t,\xi')|} \left(\tilde{u}^2(t,\xi') + \frac 1 2 \tilde{u}_x^2(t,\xi')\right) \frac {\partial x}{\partial{\bar{y}}} \frac{\partial \bar{y}}{\partial \xi'} d\xi'
\end{eqnarray*}
and to conclude we handle exactly as before for $\tilde{P}$.
\end{proof}

Let us now verify that $(\tilde{u},\tilde{v},\tilde{q})$ satisfies the system \eqref{eq_BCODE}. For  $t\ge T_{br}(\xi)$ this is clear, since the solutions have been defined by direct integration of $\eqref{eq_BCODE}$ and by Lemma \ref{Lem_Pequivtilde} the formulas \eqref{eq_defPtxi} and \eqref{eq_defPxtxi} coincide with \eqref{eq_Pxidziura} and \eqref{eq_Pxxidziura}, respectively.
Thus, we focus on the case $t < T_{br}(\xi)$ and, using the defining formulas \eqref{eq_uuu}-\eqref{eq_qqq} as well as the fact that $u$ is a weak solution of \eqref{eq_WeakCH1}-\eqref{eq_WeakCH3} which satisfies \eqref{Eq_PropL} and equations from Proposition \ref{Prop_Daf}, calculate:
\begin{eqnarray}
\partial_t \tilde{u}(t,\xi) &=& \frac {d}{dt} u(t,(\bar{y}(\xi))(t)) = -P_x(t,\bar{y}(\xi)(t))= -\tilde{P}_x(t,\xi), \nonumber\\
\partial_t \tilde{v}(t,\xi) &=& 2 \frac {d}{dt} \arctan(u_x(t,\bar{y}(\xi)(t))) \nonumber\\
&=&  \frac {2}{1+u_x^2(t,\bar{y}(\xi)(t))} \frac {d}{dt} u_x(t,\bar{y}(\xi)(t)) \nonumber\\
&=& \frac {2(u^2(t,\bar{y}(\xi)(t))-\frac 1 2 u_x^2(t,\bar{y}(\xi)(t)) - P(t,\bar{y}(\xi)(t)))}{1+u_x^2(t,\bar{y}(\xi)(t))}\nonumber \\
 &=& \frac {2\tilde{u}^2(t,\xi)- \tilde{u}_x^2(t,\xi) - 2\tilde{P}(t,\xi)}{1+\tilde{u}_x^2(t,\xi)},\nonumber\\
\partial_t \tilde{q}(t,\xi) &=& \frac {d}{dt}\left(e^{\int_0^t u_x(s,\bar{y}(\xi)(s))ds} \frac {1+u_x^2(t,\bar{y}(\xi)(t))}{1+{u}_x^2(0,\bar{y}(\xi))}\right) \nonumber\\
&=&
u_x(t,\bar{y}(t,\xi))\tilde{q}(t,\xi) + \frac {e^{\int_0^t u_x(s,\bar{y}(\xi)(s))ds}}{1+u_x^2(0,\bar{y}(\xi))} 2u_x(t,(\bar{y}(\xi))(t) \frac {d}{dt} u_x(t,\bar{y}(\xi)(t))\nonumber \\
&=& \tilde{u}_x(t,\xi) \tilde{q}(t,\xi) \left[1 + \frac{2(u^2(t,\bar{y}(\xi)(t))-\frac 1 2 u_x^2(t,\bar{y}(\xi)(t)) - P(t,\bar{y}(\xi)(t)))}{1+u_x^2(t,\bar{y}(\xi)(t))} \right] \nonumber \\
&=& \tilde{u}_x(t,\xi) \tilde{q}(t,\xi)\left[\frac{1+2\tilde{u}^2(t,\xi)-2\tilde{P}(t,\xi)}{1+\tilde{u}_x^2(t,\xi)} \right]. \label{eq_q}
\end{eqnarray}
To conclude, we observe that the final expressions for $\partial_t \tilde{v}$ and $\partial_t \tilde{q}$ are equivalent, by an algebraic manipulation, using formulas (cf. formulas (3.5)-(3.6) and (3.11)-(3.12) in \cite{BC2})
\begin{eqnarray*}
\frac {1}{1+\tilde{u}_x^2} = \cos^2(\tilde{v}/2) , \ \frac {\tilde{u}_x^2}{1+\tilde{u}_x^2} = \sin^2(\tilde{v}/2), \ 1+\cos(\tilde{v}) =  2\cos^2(\tilde{v}/2), \ \frac{\tilde{u}_x}{1+\tilde{u}_x^2}=\frac 1 2 \sin(\tilde{v}), \\  
\end{eqnarray*}
to the right hand side of \eqref{eq_BCODE} and use Lemma \ref{Lem_Pequivtilde}.

\begin{remark} 
\label{rem_every}
Observe that $\tilde{u}(t,\xi) = u(t,\bar{y}(\xi)(t))$ for every $t \in [0,T]$. Indeed, for $t \in [0,T_{br}(\xi)]$ this follows from definition \eqref{eq_uuu} and continuity of $u$. Otherwise, it suffices to note that 
$$\partial_t \tilde{u}(t,\xi) = -P_x(t,\bar{y}(\xi)(t))=\frac d {dt} u(t,\bar{y}(\xi)(t))$$ holds also for $t \ge T_{br}(\xi)$ due to Proposition \ref{Prop_Daf} and definition of $\tilde{u}(t,\xi)$ for $t\ge T_{br}(\xi)$, which consists in integrating \eqref{eq_BCODE}. 
\end{remark}

{\bf Part 3.} Since the triple $(\tilde{u},\tilde{v},\tilde{q})$ obtained above satisfies the system \eqref{eq_BCODE}, it is also a fixed point of the Picard operator $\mathcal{P}$. Moreover, by the a priori estimates in \cite[Section 4]{BC2} as well as equation \eqref{eq_C1C} the triple $(\tilde{u},\tilde{v},\tilde{q})$ belongs to the domain $\mathcal{D}$. Hence, by Theorem 4.1 in \cite{BC2}, $(\tilde{u},\tilde{v},\tilde{q})$ depends only on the initial condition.

To finish the proof, let us consider two dissipative solutions $u_1(t,x)$ and $u_2(t,x)$ satisfying $u_1(0,x)=u_2(0,x)$ for $x \in \mathbb{R}$. By the uniqueness of solution of \eqref{eq_BCODE} we have that the corresponding triples $(\tilde{u}_1,\tilde{v}_1,\tilde{q}_1)$ and $(\tilde{u}_2,\tilde{v}_2,\tilde{q}_2)$ are equal. 
The transformation of variables $x \mapsto \xi$ is bijective and is the same for $u_1(t,x)$ and $u_2(t,x)$ since it depends only on the initial condition. Hence, for every $\zeta \in \mathbb{R}$ there exists a unique $\xi$ such that $\bar{y}(\xi) = \zeta$. Consequently, denoting $\bar{y}(\xi)_1(\cdot)$ the unique forwards characteristic of $u_1$ emanating from $\bar{y}(\xi)$ and $\bar{y}(\xi)_2(\cdot)$ the unique forwards characteristic of $u_2$ emanating from $\bar{y}(\xi)$, we obtain
\begin{equation*}
\frac {d}{dt} \bar{y}(\xi)_1(t) = u_1(t,\bar{y}(\xi)_1(t)) = \tilde{u}_1(t,\xi) = \tilde{u}_2(t,\xi) = u_2(t,\bar{y}(\xi)_2(t)) = \frac {d}{dt} \bar{y}(\xi)_2(t)
\end{equation*}
where we used the definitions of $\bar{y}(\xi)_1$, $\bar{y}(\xi)_2$ and definitions of $\tilde{u}_1(t,\xi),\tilde{u}_2(t,\xi)$. Hence,
\begin{equation*}
\frac {d}{dt} ( \bar{y}(\xi)_1(t)-\bar{y}(\xi)_2(t))=0
\end{equation*}
and thus
\begin{equation*}
\bar{y}(\xi)_1(t) = \bar{y}(\xi)_2(t)
\end{equation*}
for $t$ small enough,
which means that the characteristics of $u_1(t,x)$ and $u_2(t,x)$ are identical.

Now, by Theorem \ref{Th_STT}, for all $t \ge 0$ small enough and almost every $x \in \mathbb{R}$ (i.e. $x\in S^1_t(t) \cap S^2_t(t)$, where $S_t^1$ corresponds to $u_1$ and $S_t^2$ corresponds to $u_2$) there is a unique $\zeta \in \mathbb{R}$ and a unique characteristic $\zeta(\cdot)$ of both $u_1$ and $u_2$ satisfying $\zeta(t)=x$.
Moreover, $\zeta \in S^1_t \cap S^2_t$ and there is a unique $\xi$ such that $\bar{y}(\xi)=\zeta$.
Thus, $$u_1(t,x) = u_1(t,\zeta(t)) = u_1(t,\bar{y}(\xi)(t)) = \tilde{u}_1(t,\xi).$$
Similarly,  
\begin{equation*}
u_2(t,x) = u_2(t,\zeta(t)) = u_2(t,\bar{y}(\xi)(t)) = \tilde{u}_2(t,\xi).
\end{equation*} 
Since $\tilde{u}_1(t,\xi)=\tilde{u}_2(t,\xi)$ we conclude
\begin{equation*}
u_1(t,x)=u_2(t,x)
\end{equation*}
for all $t$ small enough and $x \in S^1_t(t) \cap S^2_t(t)$. Thus $u_1(t,x)=u_2(t,x)$ a.e. and by continuity of $u_1(t,x)$ and $u_2(t,x)$ we obtain that $u_1 = u_2$ locally in time. 

\begin{remark}
Above one  can, alternatively, argue without 
Theorem \ref{Th_STT}. Indeed, the fact that all characteristics of $u_1$ and $u_2$ are unique forwards and identical implies that for every $t$ small enough and \emph{every} $x$ there exists a nonunique $\zeta \in \mathbb{R}$ such that $\zeta(\cdot)$ is a characteristic of both $u_1$ and $u_2$ and $\zeta(t)=x$. Taking any such $\zeta$ and using Remark \ref{rem_every} we obtain, as above, $u_1(t,x) = u_1(t,\zeta(t)) = u_1(t,\bar{y}(\xi)(t)) = \tilde{u}_1(t,\xi)$ for every $x \in \mathbb{R}$. Similarly, $u_2(t,x) = \tilde{u}_2(t,\xi)$ for every $x \in \mathbb{R}$, which implies $u_1(t,x) = u_2(t,x)$ for every $t$ small enough and every $x \in \mathbb{R}$.
\end{remark}

Finally, to prove global uniqueness of dissipative solutions we observe that for fixed $t_0 \ge 0$ we have $u_x \le \bar{C}$ on $[t_0,\infty) \times \mathbb{R}$ for some global constant $\bar{C}$ and handle similarly as above, using Corollary \ref{Cor_globally} instead of Theorem \ref{Th_STT}.

\end{document}